\documentclass[11pt]{amsart}
\usepackage{hyperref,pb-diagram,amssymb,amscd, graphicx, graphics, epsfig, psfrag, paralist, relsize, lmodern, enumitem}
\hypersetup{colorlinks}
\usepackage{color}
\definecolor{darkred}{rgb}{0.5,0,0}
\definecolor{darkgreen}{rgb}{0,0.5,0}
\definecolor{darkblue}{rgb}{0,0,0.5}
\hypersetup{ colorlinks,
linkcolor=darkblue,
filecolor=darkgreen,
urlcolor=darkred,
citecolor=darkblue }

\usepackage{amssymb,verbatim,graphicx}
\makeatletter 
\renewcommand\p@enumii{}
\makeatother

\newtheorem{theorem}{Theorem}[section]

\newtheorem{assumption}[theorem]{Assumption}

\newtheorem{proposition}[theorem]{Proposition}
\newtheorem{lemma}[theorem]{Lemma}

\theoremstyle{definition}
\newtheorem{definition}[theorem]{Definition}
\theoremstyle{remark}
\newtheorem{remark}[theorem]{Remark}
\newtheorem{example}[theorem]{Example}

\newcounter{notes}
{\end{list}}

\makeatletter
\@addtoreset{theorem}{section}
\makeatother
\newcommand\A{\mathcal{A}}

\newcommand\mC{\mathcal{C}}

\newcommand{\K}{\mathcal{K}}

\newcommand{\V}{\mathcal{V}}

\newcommand{\U}{\mathcal{U}}

\newcommand{\F}{\mathcal{F}}

\newcommand{\R}{\mathbb{R}}

\newcommand{\C}{\mathbb{C}}

\newcommand{\Z}{\mathbb{Z}}
\newcommand{\Q}{\mathbb{Q}}
\newcommand\G{\mathcal{G}}
\renewcommand\H{\mathcal{H}}

\newcommand{\ddt}{\frac{d}{dt}}

\newcommand{\pps}{\frac{\partial}{\partial s}}

\newcommand{\ppnu}{\frac{\partial}{\partial \nu}}

\renewcommand{\P}{\mathbb{P}}

\newcommand{\tk}{\tilde k}
\newcommand{\tu}{\tilde u}
\newcommand{\tU}{\tilde U}
\newcommand{\tA}{\tilde A}
\newcommand{\tSig}{\tilde \Sigma}

\newcommand{\dvol}{\on{dvol}}
\newcommand{\vol}{\on{vol}}
\newcommand{\Euc}{\on{Euc}}

\newcommand\lie[1]{\mathfrak{#1}}
\renewcommand{\k}{\lie{k}}

\newcommand{\g}{\lie{g}}

\newcommand{\on}{\operatorname}

\newcommand{\Ad}{ \on{Ad} }

\newcommand\Vect{\on{Vect}}

\newcommand\Hom{\on{Hom}}

\newcommand{\im}{ \on{Im}}
\newcommand\Id{\on{Id}}
\newcommand\onD{\on{D}}

\newcommand{\hra}{\hookrightarrow}
\newcommand{\weakto}{\rightharpoonup}

\renewcommand{\d}{{\on{d}}}
\newcommand{\loc}{{\on{loc}}}

\newcommand{\orb}{{\on{orb}}}
\newcommand{\ol}{\overline}

\newcommand{\delbar}{\ol{\partial}}
\newcommand\bs{\backslash}

\newcommand\Sig{\Sigma}
\newcommand\sig{\sigma}
\newcommand\eps{\epsilon}
\newcommand\Om{\Omega}
\newcommand\om{\omega}

\newcommand{\lan}{\langle}
\newcommand{\ran}{\rangle}
\newcommand{\hh}{{\frac{1}{2}}}

\newcommand{\tqq}{{\frac{3}{4}}}

\newcommand{\fix}{n}
\renewcommand{\ss}{{\operatorname{ss}}}
\newcommand{\dual}{\vee}
\newcommand\Mod[1]{\lVert #1 \rVert}
\newcommand\qu{/\kern-.7ex/} 
\newcommand\mO{\mathcal{O}}

\begin{document}
\author{S. Venugopalan}

 \author{C. Woodward}

\subjclass[2010]{53D06}

\title{Classification of affine vortices}

\begin{abstract}
We prove a Hitchin-Kobayashi correspondence for vortices on the
complex affine line with K\"ahler target, which generalizes a result
of Taubes \cite{taubes1980}, see also \cite{JT}, for the case of a
line target.  More precisely, suppose a compact connected Lie group
$K$ acts on a K\"ahler manifold $X$ with proper moment map so that
stable=semistable for the action of the complexified Lie group $G$ and
$X$ is equivariantly convex at infinity. Then, for some sufficiently
divisible integer $\fix$, there is a bijection between gauge
equivalence classes of $K$-vortices with target $X$ modulo gauge and
isomorphism classes of maps from the weighted projective line
$\P(1,\fix)$ to $X/G$ that map the stacky point at infinity $\P(\fix)$
to the semistable locus of $X$.  The results allow the construction
and partial computation of the quantum Kirwan map in Woodward
(\cite{W:qkirwan1}, \cite{W:qkirwan2} and \cite{W:qkirwan3}), and play a role in the conjectures of Dimofte,
Gukov, and Hollands \cite{dimofte:vortex} relating vortex counts to
knot invariants.
\end{abstract}

\maketitle

\tableofcontents 

\parskip .05in

\section{Introduction}\label{sec:intro}

A {\em vortex} in geometric analysis is a pair consisting of a
connection and bundle section which satisfies an equation involving
the curvature and a non-linear (usual quadratic) term depending on the
section, as well as a first order equation for the section.  Vortices
arise naturally in several areas of mathematical physics and geometry;
our interest in them arises from their interpretation as equivariant
generalizations of pseudoholomorphic maps.  Our goal in this paper is
to classify vortices, in the special case that the domain is the
complex line.  In this setting, Taubes \cite{taubes1980} 
and Jaffe-Taubes \cite{JT} provide a classification in the first
non-trivial case of finite energy vortices with circle group $K =U(1)$
and target space $X = \C$.  In this paper we generalize their
classification to arbitrary compact connected Lie groups $K$ and fiber
bundles whose fiber is a K\"ahler manifold $X$ with a Hamiltonian
$K$-action. Here $X$ is either compact or is equivariantly convex at
infinity with a proper moment map.  More precisely we classify {\em
  affine vortices} with target $X$, consisting of pairs $(A,u)$ where
$A$ is a connection on the principal bundle $P=\C\times K$, $u:\C \to
P(X) := (P \times X)/K $ is a holomorphic section with respect to the
complex structure defined by $A$ and $(A,u)$ satisfies the vortex
equation:
\begin{equation} \label{ve} *F_A+\Phi(u)=0. \end{equation}
Here $F_A$ is the curvature of $A$, $\Phi: X \to \k^\dual\simeq \k$ is
the moment map on $X$ for the $K$-action, and $*$ is the Hodge star
for the standard metric on $\C$.

In the special case of the action of the circle group on the complex
line, Taubes \cite[Theorem 1]{taubes1980}) and Jaffe-Taubes
\cite[Section III, Theorem 1]{JT} show that the gauge-equivalence
class of a finite energy vortex is completely determined by the zeros
of the section. It follows that the moduli space of gauge equivalence
classes of finite energy vortices is a symmetric product. This gives a
description of the space of vortices with target $\C$ in terms of
holomorphic data which from a more modern perspective may be viewed as
a {\em Hitchin-Kobayashi correspondence}.  The first example of such a
correspondence by Narasimhan and Seshadri \cite{NS} shows that stable
holomorphic bundles over a Riemann surface correspond to irreducible
unitary representations of the fundamental group.  Donaldson
\cite{Do:NS} reproves the Narasimhan-Seshadri theorem in a
differential-geometric setting, replacing irreducible unitary
representations by an equivalent object - the minima of the Yang-Mills
functional.  The extension of this result to higher dimensional base
manifolds involved replacing a Yang-Mills connection by a
Hermitian-Einstein connection. The Hitchin-Kobayashi correspondence in
this case states that a holomorphic vector bundle admits a
Hermitian-Einstein connection if and only if it is stable. This was
proved for compact K\"ahler surfaces in \cite{Do:ASD} and for general
compact K\"ahler manifolds in \cite{UY}.

In the case of vortices the corresponding holomorphic objects are
holomorphic bundles over compact K\"ahler manifolds with additional
data, for example a prescribed holomorphic section.  Vortices are the
zeros of the {\em vortex functional} given by the norm-square of the
left hand side of \eqref{ve}. Bradlow's paper
\cite{Brad:stability} defines a stability condition for such objects
and relates it to the existence of zeros of the vortex functional.
Results in Bradlow \cite{Brad:stability} are used to investigate the
moduli space of finite energy vortices in \cite{Brad:moduli1},
\cite{Brad:moduli2} and \cite{Brad:vortices}.  Garcia-Prada
\cite{oscar1}, \cite{oscar2}, \cite{oscar3} provides a different
approach to similar results via dimensional reduction, achieving in
some cases more general results.  Mundet \cite{Mund} generalizes these
results by allowing the fiber to be a K\"ahler Hamiltonian manifold.
All these Hitchin-Kobayashi results are infinite-dimensional versions
of the abstract setting laid out in Kempf-Ness \cite{KN} and Kirwan
\cite{Kirwan}, in which the main idea is that the symplectic quotient
coincides with the geometric invariant theory (git) quotient.

Our main result is a bijection between
finite energy affine vortices with bounded image and a moduli space of holomorphic
objects defined as follows: 
Let $G$ be a connected complex reductive group whose maximal compact subgroup is $K$. 
By a {\em $G$-gauged holomorphic map} we mean
a pair $(P_\C,u)$ where $P_\C \to \P(1,n)$ is a holomorphic principal
$G$-bundle over the weighted projective line $\P(1,n)$ and $u: \P(1,n) \to P_\C \times_G X$ is a section. Complex
gauge transformations are $G$-automorphisms of the underlying smooth
bundle; they act on $G$-gauged maps by pulling back the complex structure
and section. We say that a $G$-gauged holomorphic map is {\em stable} if
the point at infinity is mapped by $u$ to the semistable locus
$X^{\ss}$, i.e., $u(\infty)$ takes values
in $(P_\C \times_G X^{\ss})|_\infty$. The stability condition we consider here is surprisingly
simple because we want our holomorphic objects to correspond to
unitary objects on {\em affine} curves.  From a $G$-gauged holomorphic map
$(P_\C,u)$, we produce a unitary object by choosing a section $\sig:
\P(1,n) \to P_\C/K$. This reduces the structure group to the maximal
compact subgroup $K$, i.e. there is a principal $K$-bundle $P
\subseteq P_\C$ such that $P_\C=P \times_K G$.  The complex structure
on $P_\C$ corresponds to a choice of $K$-connection $A$.

\begin{theorem}{\rm (Classification of affine vortices)} \label{thm:premain}  
  Let $X,G,K$ be as above.  Suppose the $G$-action on $X^\ss$ has finite stabilizers and $\fix$ is a positive integer such that
  for any $x \in X^{\ss}$, the order of the stabilizer group $|G_x|$
  divides $\fix$.  For any stable $G$-gauged map $(P_\C,u)$, there exists
  a unique reduction of the structure group $\sig:\C \to P_\C/K$, so that if $A$ is the resulting connection, then $(A,u)|_\C$ is a symplectic vortex. The reduction $\sig$ extends over $\P(1,n)$ in
  the sense of Section \ref{sec:outside} below.  The map $[(P_\C,u)]
  \mapsto [(A,u)|_\C]$ defines a bijection between complex gauge
  equivalence classes of stable gauged holomorphic maps $(P_\C,u)$
  from $\P(1,\fix)$ to $X$ and gauge equivalence classes of 
  finite energy affine $K$-vortices with bounded image in $X$.
\end{theorem}

The moduli space of $G$-gauged holomorphic maps in the theorem has a
convenient stack-theoretic interpretation. The weighted projective
line is a Deligne-Mumford stack, or in other language, an orbifold. Recall that if $C$ is
a complex curve, a morphism from $C$ to the {\em quotient stack} $X/G$ consists of a principal $G$-bundle over $C$ together with a
section of the associated fiber bundle $P \times_G X$.  The quotient
stack $X/ G$ contains as a proper open substack the geometric
invariant theory quotient $X \qu G$, here defined as the
stack-theoretic quotient of the semistable locus by the action of
$G$. The same is true when $C$ is replaced by $\P(1,n)$ (see
\cite{Lerm:stack}). The moduli space of stable $G$-gauged holomorphic maps
is the coarse moduli space of the open substack of $\Hom(\P(1,n),X/G)$
that satisfies $u(\P(n)) \subset X\qu G$.  In the case that $G$ is a
torus acting on a finite dimensional complex vector space $X$, bundles
on $\P(1,\fix)$ and sections of the associated vector bundle can be
classified explicitly, see Proposition 1.2 below.

A recent paper of Xu \cite{Xu:vortex} complements our work. It proves
similar results as this paper for $U(1)$-vortices with fiber $X=\C^m$
using results of \cite{Brad:stability}. It also shows a correspondence
between compactifications of the space of affine vortices modulo gauge
on the one side and the space of gauged holomorphic maps over $\P^1$,
that are semistable at $\infty$.

The Hitchin-Kobayashi correspondence for affine vortices in Theorem
\ref{thm:premain}
is partly motivated by a certain quantization of the Kirwan map that
arises in the study of Gromov-Witten invariants of geometric invariant
theory quotients.  Namely Kirwan \cite{Kirwan} constructs a map from
the equivariant cohomology $H_G(X,\Q)$ to the cohomology of the
quotient $H(X \qu G, \Q)$.  A Gromov-Witten generalization of
\cite{Kirwan} called the quantum Kirwan map, suggested by Gaio-Salamon
\cite{GS}, maps the equivariant quantum cohomology of $X$ to the
quantum cohomology of the quotient $X \qu G$.  Salamon and Ziltener
\cite{Zilt:thesis} suggested to define the quantum Kirwan map by a
count of affine vortices.  The papers of the second author
\cite{W:morphism}, \cite{W:qkirwan1}, \cite{W:qkirwan2} and \cite{W:qkirwan3} generalize the manifolds for
which this map is defined, by removing monotonicity and asphericity
assumptions on $X$. Also, the group action can have finite stabilizers
on the zero-level set of the moment map, i.e. the symplectic quotient
can be an orbifold. This paper is part of that project. It provides an
algebraic description for the moduli space of vortices, and from
there, the quantum Kirwan map can be defined using Behrend-Fantechi
machinery. A very important result in this context is a
compactification for the space of affine vortices modulo gauge proved
by Ziltener (\cite{Zilt:thesis}, \cite{Zilt:QK}).

An additional recent motivation arises from {\em knot-invariants via
  vortex counting} conjectures of Dimofte, Gukov, and Hollands
\cite{dimofte:vortex}.  In these conjectures, the equivariant index
(defined via localization) of the moduli space of affine vortices is
conjectured to be a certain knot invariant.  Our results allow the
identification of the moduli space of affine vortices with the {\em
  quasimap} spaces discussed in, for example, Bertram,
Ciocan-Fontanine, and Kim \cite{be:tw}.  The space of matrix-valued
vortices in Example \ref{matrix} appears as the relevant space of
vortices for a torus knot in the vortex counting conjectures of
\cite{dimofte:vortex}.

The proof of the Hitchin-Kobayashi correspondence for affine vortices
uses a minimization technique for the vortex functional.  In
particular the first author proved in \cite{Venu} that the heat flow
of the functional
$$(A,u) \mapsto \Mod{*F_A + \Phi(u)}_{L^2(\Sig)}$$
exists for all time on any compact Riemann surface $\Sig$, possibly with
boundary.  Using this result, the vortex equation on an affine line
can be solved on a succession of annuli of increasing sizes that
exhaust the domain.  We prove that this sequence of vortices converges
to a solution on the whole plane. This process involves a number of
complications, such as proving that the sequence of solutions on the
balls converges without bubbling.

We give some examples of the main result in the case of torus actions
on vector spaces. By Theorem \ref{thm:main}, given a finite energy affine vortex $(A,u)$ with bounded image,
it corresponds to a gauged holomorphic map over a principal $K$-bundle $P
\to \P(1,n)$ in a sense defined in Section 3.1.  The characteristic class $[P]$ of $P$ in $H_2(BK,\Q)$
is a topological invariant of $(A,u)$ (see Remark
\ref{rem:topinv}). If $K$ is a torus then $H_2(BK,\Q)$ is isomorphic
to the subset $\{\lambda \in \k\,|\, \exists n \in \Z, \ e^{2\pi \fix
  \lambda}=\Id\}$ of rational elements $\k$ and we call the image of
$[P]$ in $\k$ the {\em degree} of the vortex. In the following
Proposition, we classify affine vortices of fixed degree.

\begin{proposition}{\rm (Classification of affine vortices in the
    toric case)} \label{prop:toric} Suppose that $G$ is a complex
  torus acting on a complex vector space $X=\C^k$ with weights
  $\nu_1,\ldots,\nu_k \in \g^\dual$ contained in an open half space,
  and spanning $\g^\dual$.  Then, under the bijection in Theorem
  \ref{thm:premain} affine vortices of degree $d \in \k \subset \g$
  correspond to isomorphism classes of tuples of polynomial maps $u =
  (u_1,\ldots,u_k): \C \to X$ satisfying
  \begin{enumerate}
  \item the degree of $u_j$ is at most $\lan \nu_j,d \ran$ for each $j
    = 1,\ldots, k$; and
  \item if
$$ u(\infty) := \left( u_j(\infty) := \begin{cases} u_j^{ (\lan \nu_j, d \ran) }/ 
    \lan \nu_j , d \ran !  & \lan \nu_j, d \ran \in \Z \\
    0 & \text{otherwise} \end{cases} \right)_{j=1}^k $$
denotes the vector of leading order coefficients (here $u_j^{(\lan \nu_j,d\ran)}$ denotes the $ \lan \nu_j, d
\ran$-th derivative of $u_j$) with integer exponents, then $u(\infty) \in
X^{\ss}$.
\end{enumerate}
Two such tuples are isomorphic if they are related by the action of
$G$.
\end{proposition}
We defer the proof to the end of Section \ref{sec:last}.

\begin{example} \label{matrix}
  \begin{enumerate}
  \item
    {\rm (Linear vortices)} If $X = \C$ with $\nu_1 = 1$, then affine
    vortices of class $d$ correspond to polynomials of degree exactly
    $d$ up to the action of scalar multiplication, hence classified by
    their zeroes.  This recovers the Taubes \cite{taubes1980},
    \cite{JT} result.
  \item {\rm (Matrix-valued vortices and Weil's torsion quotients)} If
    $X = M_n(\C)$, the space of $n\times n$ matrices, and $G = GL_n$
    acts by left multiplication, then the semistable locus consists of
    invertible matrices and the action on the semistable locus is
    free. Theorem \ref{thm:main} gives a classification of vortices
    according to the following data: By Grothendieck's theorem
    \cite{gr:cl}, any holomorphic vector bundle on $\P^1$ splits as a
    sum of line bundles
$$ P \times_{GL_n(\C)} \C^n \cong \mO_{\P^1}(\lambda_1) \oplus \ldots \oplus
\mO_{\P^1}(\lambda_n), \quad \lambda_1 \ge \ldots \ge \lambda_n $$
where $\mO_{\P^1}(\lambda)$ is the $\lambda$-th tensor power of the
hyperplane bundle.  The associated $X$ bundle is then
$P(X) = \mO_{\P^1}(\lambda_1)^{\oplus n} \oplus \ldots \oplus
\mO_{\P^1}(\lambda_n)^{\oplus n} .$
A section $u$ of $P(X)$ may be viewed as a matrix-valued function on
$\C$.  The semistability condition at infinity is then the condition
that the leading order terms of $u$ form an invertible matrix.  Thus
$u$ defines a morphism of sheaves which is generically an isomorphism,
providing a connection to Weil's scheme of {\em torsion quotients}
considered in \cite{be:tw}, \cite{Biswas:mero}, which \cite{be:tw}
mentions was considered by Weil to be the non-abelian analog of the
symmetric product of the curve.
\end{enumerate}
\end{example}

We briefly sketch the contents of the paper.  Section
\ref{sec:background} defines gauged holomorphic maps and extends the
definitions to the case when the base manifold $\Sig$ is an orbifold
$\P(1,\fix)$. Section \ref{sec:outside} gives an analytic version of
the main theorem and proves the first part of the main theorem
\ref{thm:main}. Section \ref{sec:last} proves the second part. The
proof relies on removal of singularities at infinity for finite energy
vortices with bounded image which is proved in Section
\ref{sec:orbdecay}.

{\bf Acknowledgments:} We thank Fabian Ziltener for useful
discussions. 
S.V. was a post-doctoral fellow at Tata Institute of Fundamental Research, Bombay at the time this paper was written. C.W. was partially supported by NSF grant DMS 1207194 and a Simons Fellowship. We also acknowledge HIM, Bonn for hosting us for a week, when some of the work on this paper was carried out. Finally, we profusely thank the referee for reading the paper carefully and providing suggestions that improved the paper.
\section{Background}\label{sec:background}

In this section we introduce basic notation for gauged holomorphic
maps and vortices, especially gauged holomorphic maps on weighted
projective lines.

\subsection{Hamiltonian manifolds}

We introduce the following notation for group actions.
Let $G$ be a complex reductive Lie group with maximal compact subgroup
$K$, so that $G$ is the complexification of $K$. Let $(X,\om,J)$ be a
K\"{a}hler manifold with symplectic structure $\om$ and complex
structure $J$ on which $G$ acts holomorphically and $K$ acts
symplectically.  A {\em moment map} is a $K$-equivariant map $\Phi: X
\to \k^\dual$ such that
\begin{equation}\label{eq:mommap}
\iota(\xi_X) \omega = \d \lan \Phi, \xi \ran , \ \forall \xi \in \k ,
\end{equation}
where
$\xi_X(x)=\ddt \exp(t\xi)x|_{t=0} \in \Vect(X)$ is given by the infinitesimal action of $\xi$ on $X$.
The action of $K$ is {\em Hamiltonian} if there exists a
moment map $\Phi : X \to \k^\dual$. Since $K$ is compact, $\k$ has an
$\Ad$-invariant metric. We fix such a metric and identify $\k \simeq
\k^\dual$, so the moment map becomes a map $\Phi : X \to \k$.  We
assume $X$ is equipped with a Hamiltonian action and fix the moment
map.  In the rest of the paper, we assume:
\begin{assumption} \label{ass:freeaction} The 
  $G$-action on $G\Phi^{-1}(0)$ has finite stabilizers.
\end{assumption}
If the underlying complex structure of $X$ has the structure of a polarized projective or affine $G$-variety, the
git quotient is defined by $X \qu G :=
X^{\ss}/\sim$ where $X^{\ss}$ is the semi-stable locus and $\sim$ is
the orbit closure relation. For details of the construction, refer to
\cite{Mumford} for the projective case and \cite{King} for the affine
case.  A polarization of $X$ determines a moment map. By the Kempf-Ness \cite{KN} theorem and its generalization
to affine varieties (\cite[Theorem 6.1]{King}), the 
quotient $X \qu G$ is homeomorphic to the symplectic quotient
$\Phi^{-1}(0)/K$. 
Assumption \ref{ass:freeaction} guarantees that the $K$ action on
$\Phi^{-1}(0)$ has finite stabilizers, and that
$X^{\ss}=G\Phi^{-1}(0)$. In case $X$ is a K\"ahler manifold without an algebraic structure, we define the semistable locus as $X^\ss :=G\Phi^{-1}(0)$.

\subsection{Connections and curvature}

We begin with basic notions of connections, curvature, and their
behavior under gauge transformations. Let $\Sig$ be a Riemann surface,
equipped with a complex structure $j_\Sig:T\Sig \to T\Sig$ and a
volume form $\om_\Sig \in \Om^2(\Sig)$ that induces a metric. Let $P
\to \Sig$ be a principal $K$-bundle. A {\em connection} is a
$K$-equivariant one-form $A \in \Om^1(P,\k)^K$, that satisfies
$A(\xi_P)=\xi$ for $\xi \in \k$. The {\em space of connections}
$\A(P)$ is an affine space modelled on $\Om^1(\Sig, P(\k))$, where
$P(\k)=P \times_K \k$ is the {\em adjoint bundle}. In case $P$ is the
trivial bundle $\Sig \times K$, there is a trivial connection $\onD$,
and the adjoint bundle has a trivialization $P(\k) \simeq \Sig \times
\k$. Then, the space of connections is
$$\A(P)=\onD +\, \Om^1(\Sig, \k).$$ 
The {\em curvature} of a connection $A$ is a
 two-form $F_A \in \Om^2(\Sig, P(\k))$. In particular, on a trivial
 bundle, for a connection $A=\onD+a$,
$$F_A:= \d a + [a \wedge a]/2 \in \Om^2(\Sig,\k).$$
The curvature varies with the connection as
\begin{equation}
\label{eq:fdepa}
F_{A+ta}=F_A+t \d_A a+\frac{t^2}2[a \wedge a]
\end{equation}
where $\d_A$ is the associated covariant derivative on the adjoint
bundle. A {\it gauge transformation} is an automorphism of $P$ that is an
equivariant bundle map $P \to  P$. Alternatively, it is a section of the
bundle $P \times_K K \to  \Sig$, where $K$ acts on itself by
conjugation. The group of gauge transformations on $P$ is denoted
$\K(P)$. On the trivial bundle $\Sig \times K$, an element in $\K(P)$ is a map
from $\Sig$ to  $K$. An element $k \in \K(P)$ acts on a connection $A=\onD+a$ as
$$k(A)= \onD+(k \d k^{-1} + \Ad_ka).$$
Differentiating, we see that the infinitesimal action of $\xi:\Sig \to
\k$ on $A$ is $-\d_A\xi$.

\subsection{Complex gauge transformations}
We introduce notation for complex gauge transformations.  Associated
to the principal $K$-bundle $P \to \Sig$, we have a principal
$G$-bundle $P_\C:=P \times_K G$ on $\Sig$.  A {\em complex gauge
  transformation} is a $G$-equivariant bundle automorphism $P_\C \to
P_\C$. The group of complex gauge transformations is denoted by
$\G(P)$. It is equivalent to the space of sections $g:\Sig \to P \times_K G$, where
$K$ acts on $G$ by conjugation. Recall that
\begin{equation}\label{eq:complexgroupiso}
  K \times \k \to G, \quad 
  (k,s) \mapsto ke^{is}
\end{equation}
is a diffeomorphism (see Helgason
  \cite[VI.1.1]{Helg2}). So, a complex gauge transformation $g$ can be
written as $g=ke^{i\xi}$, where $k \in \K(P)$ and $\xi \in
\on{Lie}(\K(P)) = \Gamma(P(\k))$. We next explain the action of
$\G(P)$ on the space of connections $\A(P)$.

Unitary connections determine holomorphic structures on
fiber bundles as follows. For any $K$-manifold $X$ we have an associated fiber bundle $P(X)$ over
$\Sigma$, and the connection $A$ defines a connection on $P(X)$ and
hence a covariant derivative 
$$\d_A : \Gamma(P(X)) \to \bigcup_{u \in
  \Gamma(P(X))} \Om^1(\Sig, u^* T^{\on{vert}}(P(X))) ,$$ 
where $T^{\on{vert}}(P(X))$ is the vertical part of the tangent
bundle.  For example on the bundle $\Sig \times X$, writing
$A=\onD+a$,
\begin{equation}\label{eq:dau}
\d_Au := \d u + a_u \in \Om^1(\Sig, u^*TX).
\end{equation}
At a point $z \in \Sig$, $a_u(z)$ is the infinitesimal action of
$a(z)$ at $u(z)$.  If the fiber $X$ is also a complex manifold so that
the $K$-action is holomorphic, then one obtains a holomorphic
structure on the associated fiber bundle $P(X)$ by defining the the
$\delbar_A$ operator by
$$\delbar_A:\Gamma(\Sig,P(X)) \to \Om^{0,1}(\Sig, u^*T^{\on{vert}}P(X)), \quad u \mapsto (\d_Au)^{0,1}.$$
In particular, taking $X=G$ produces a holomorphic $G$-bundle $P_\C$. Let $\mC(P)$ denote the space of holomorphic structures on $P_\C$.
The above construction yields a map from $\A(P)$ to $\mC(P)$.
Conversely, given a holomorphic bundle $P_\C$, a choice of a section
$\sig: \Sig \to P_\C/K$ gives a principal $K$-bundle $P$ by pullback
of the bundle $P_\C \to P_\C/K$, so that $P$ is naturally a submanifold of
$P_\C$.  The intersection $TP \cap J(TP)$ defines a connection in $TP$
\cite{Singer}. The correspondence between connections and holomorphic structures
 gives an infinitesimal isomorphism
$$ T_A \A =\Om^1(\Sig, P(\k)) \to T_C\mC=\Om^{0,1}(\Sig, P(\g)), \quad a \mapsto a^{0,1}.$$
The complex structure on $\mC$ pulls back to a complex structure on
$\A$ given by $J_\A a= a \circ j_\Sig$. 

The identification between complex structures and connections is
equivariant for gauge transformations.  The group of complex gauge
transformations $\G(P)$ acts on the space of holomorphic structures
$\mC$ on $P_\C$ by pull-back, and pulling back to $\A(P)$ one obtains
an action of $\G(P)$ on $\A(P)$ which extends the action of unitary
gauge transformations $\K(P)$.  For any $\xi \in \Gamma(\Sig,P(\k))$,
the infinitesimal action of $i\xi$ on $A$ is $-\d_A\xi \circ j_\Sig$.
Since $\Sig$ is a Riemann surface, the infinitesimal action of $\xi$
can be rewritten as
\begin{equation}\label{eq:infgcaction}
 (i\xi)_{\A(P)}= -\d_A\xi \circ j_\Sig=*\d_A\xi.
\end{equation}

\subsection{Gauged holomorphic maps}

Roughly speaking ``gauging'' an object means introducing a connection
into the picture. 
A {\em gauged holomorphic map}
$(A,u)$ from $P$ to $X$ consists of a connection $A$ and a section
$u$ of $P(X)$ that is holomorphic with respect to $\delbar_A$. 
The
space of gauged holomorphic maps from $P$ to $X$ is denoted
$\H(P,X)$. 
We remark that here we are talking about $K$-gauged holomorphic maps, which will be the default for this paper, as opposed to the $G$-gauged holomorphic map defined in Section \ref{sec:intro}.
A {\em symplectic vortex} is a
gauged holomorphic map that satisfies
\begin{equation}\label{eq:vortex}
F_{A,u}:=*F_A +\Phi(u)=0.
\end{equation}
In the case when $\Sig=\C$ is equipped with the standard Euclidean
metric, a solution of \eqref{eq:vortex} is called an {\em affine
  vortex}.
The {\em energy} of a gauged holomorphic map
  $(A,u)$ is
$$E(A,u):= \int_\Sig (|F_A|^2+|\d_Au|^2+|\Phi \circ u|^2) \om_\Sig.$$
For $\Sig=\C$, finite energy symplectic
vortices with bounded image have good asymptotic properties (see \cite[Section 11]{GS},
\cite{Zilt:QK}).
 %
The complexified gauge group acts on $\H(P,X)$ diagonally: $g(A,u):= (g(A),gu)$.
This action preserves holomorphicity (see \cite{Venu}) but not the
vortex equation unless the gauge transformation is unitary. 
From the
holomorphic viewpoint, a $K$-gauged holomorphic map $(A,u)$ on a Riemann
surface $\Sig$ consists of a $G$-gauged holomorphic map $(P_\C,u)$ and a section $\sig:\Sig \to P_\C/K$. 
The section $\sig$ provides a reduction of the structure group $G$ to $K$ and yields a $K$-bundle $P \subset P_\C$. The holomorphic structure on $P_\C$ determines a connection $A$ on $P$. This point of view
is used in stating Theorem \ref{thm:premain}.

\subsection{Gauge theory on weighted projective lines}\label{sec:orbibundle}
In the Hitchin-Kobayashi correspondence we wish to prove, gauged holomorphic maps
on orbifolds play an important role. 
For orbifolds, we
follow classical definitions of Satake \cite{Satake}. In Satake's definition, composition of morphisms is ambiguous, but that issue does not arise in our paper.
 Gauged holomorphic maps on
orbifolds are similar to $J$-holomorphic curves on orbifolds described
in \cite{ChenRuan}. We will not repeat full definitions here, but just
describe what the definitions give in the specific case that the
orbifold is the weighted projective line.
We begin with some notation for weighted projective 
lines.  For $\fix$ a positive integer, $\P(1,\fix)$ is the quotient of
$\C^2 - \{ 0 \}$ by the action of $\C^*$ with weights $1,\fix$.
It is covered by two orbifold charts, $\tU_1$ and $U_2$ where
$\tU_1=U_2 = \C$ mapping to $\P(1,n)$ as $z \mapsto  [(1,z)]$ and $z \mapsto [(z,1)]$ respectively. 
The equivalences are :
\begin{equation}\label{eq:charts}
\begin{split}
z \sim e^{2\pi i/\fix}z& \quad \quad  z \in \tU_1\\
z^{-\fix} \sim w& \quad \quad 0 \neq w  \in U_2, 0 \neq z \in \tU_1.
\end{split}
\end{equation}
We refer to $\tU_1/\sim$ as $U_1$. We denote by $\sig_\fix : \tU_1
\to \tU_1, \ z \mapsto e^{2\pi i/\fix}z$ the diffeomorphism
giving the action of the generator of $\Z_\fix$. 
The orbifold $\P(1,\fix)$ has smooth locus given by a copy of $\C$,
and if $\fix \neq 1$ has an ``orbifold singularity'' at $\infty$.
More precisely, the orbifold point is the quotient $\P(\fix)$ of
$\C^*$ by $\C^*$ acting with weight $\fix$.  As a groupoid this is
equivalent to the quotient of a point by $\Z_\fix$, that is,
$B\Z_\fix$.

Principal bundles on the weighted projective line can be defined via the clutching construction as follows. For a more general definition of bundles on orbifolds, see \cite{Ruan:orbifold}.
An orbifold principal $K$-bundle $P$ on $\P(1,n)$ is given by transition maps
$\mu:\tU_1 \to K$ and $\tau: \tU_1 \bs\{0\} \to K$ that satisfy $\tau(\sig_\fix(z))=\tau(z)\mu(z)^{-1}$. 
$$P:=((\tU_1 \times K) \bigsqcup (U_2 \times K))/\sim,$$
and the equivalence $\sim$ is given by
\begin{equation}\label{eq:equivbund}
\begin{split}
(z,h) &\sim (\sig_\fix z, \mu(z)h) \qquad (z,h) \in \tU_1 \times K\\
(z,h) &\sim (w,\tau(z)h) \qquad 0 \neq z \in \tU_1, w \in U_2, w=\frac 1 {z^\fix}, h \in K.
\end{split}
\end{equation}
Two
orbifold bundles over $\P(1,n)$ given by transition functions
$(\mu_0,\tau_0)$ and $(\mu_1,\tau_1)$ are isomorphic if there exist
smooth functions $\phi_1:\tU_1 \to K$ and $\phi_2:U_2 \to K$
satisfying 
\begin{equation}\label{eq:orbibundleiso}
\begin{split}
\phi^{-1}_1(\sig_\fix(z)) \mu_0(z) \phi_1(z)&= \mu_1(z) \quad \forall z \in \tU_1\\
\phi_2^{-1}(w)\tau_0(z)\phi_1(z)&=\tau_1(z) \quad \forall z \in \tU_1, w \in U_2, w=z^{-\fix}.
\end{split}
\end{equation}
Note that the fiber over the singular point $[0] \in U_1$ may not be
$K$, but rather $K/\Z_m$, where $m$ is the order of $\mu(0)$. We
remark in the case $\fix=1$, $\mu$ can be chosen to be $\Id$ and
$\tau:\C^\times \to K$ is the standard transition function.  

The
clutching construction for bundles above is related to the
classification of principal bundles up to isomorphism.  In the case
without singularities, the set of principal $K$-bundles $P \to \P^1$
is in bijection with $\pi_1(K)$. This can be seen as follows: the
deformation retract of the transition map $\C^*\to K$ is a loop $S^1
\to K$, whose homotopy class determines the bundle. The loop $S^1 \to
K$ can be deformed to a geodesic loop $\theta \mapsto e^{\lambda
  \theta}$, where $\lambda \in \k$ satisfies $e^{2\pi \lambda}=\Id$.
In the case of orbifold singularities, given a bundle $P \to
\P(1,\fix)$, the transition maps would now produce a geodesic path
$\theta \in [0,2\pi] \mapsto e^{\lambda \theta}$, $\lambda \in \k$
satisfies $e^{2\pi \fix \lambda}=\Id$, see Remark
\ref{rem:orbibundleinv}\eqref{part:stdform}.  The isomorphism type of
the bundle $P$ is determined by the homotopy class of the {\em
  classifying path} $\theta \mapsto e^{\lambda \theta}$, that is,
deformations keeping the endpoints $\Id$ and $e^{2\pi \lambda}$ fixed
or by applying $\Ad_k$ to the path for some $k \in K$.  In this way
one obtains a bijection between isomorphism classes of $K$-bundles
over $\P(1,\fix)$ and elements of $\exp( 2\pi i \fix \cdot)^{-1}(1)$,
up to conjugacy.

Connections on principal bundles can be described via their
restriction to trivialization associated to the clutching
construction.  A connection on $P \to \P(1,\fix)$ is given by
connections on trivializations $\tU_1 \times K$ and $U_2 \times K$
that satisfy the equivalences \eqref{eq:equivbund}:
\begin{enumerate}
\item The connection $A|_{\tU_1}$ satisfies
  \begin{equation}\label{eq:znsymmetry}
    \sig_\fix^*A=\mu(A),
  \end{equation}
  where $\mu$ acts on $A$ as a gauge transformation.
\item By the above condition, $\sig_\fix^*(\tau(A))=\tau(A)$ on $\tU_1
  \bs \{0\}$, so $\tau(A)$ descends to a connection on $\tU_1 \bs
  \{0\}/\Z_\fix$. We require that this descended connection is
  $A|_{U_2 \bs \{0\}}$.
\end{enumerate}
We define gauge transformations in terms of the canonical
atlas.  A { gauge transformation} $k$ on $P$ consists of $\tilde
k_1=k|_{\tU_1}:\tU_1 \to K$ and $k_2=k|_{U_2}:U_2 \to K$ satisfying
the equivalences \eqref{eq:equivbund}:
\begin{enumerate}
\item $\sig_\fix^*k_1= \mu k_1 \mu^{-1}$ and
\item $\tau k_1|_{\tU_1 \bs \{0\}}\tau^{-1}$ descends to $k_2$. 
\end{enumerate}
This description shows that the set of gauge transformations on a $P \to \P(1,n)$ forms a group $\K(P)$ under composition.

Finally we describe gauged holomorphic maps on weighted projective
lines.  Let $P \to \P(1,\fix)$ be a principal $K$-bundle. A {\em
  gauged holomorphic map} $(A,u)$ from $\P(1,\fix)$ to $X$ consists of
gauged holomorphic maps on the bundles $\tU_1 \times K$ and $U_2
\times K$ that satisfy the equivalence conditions \eqref{eq:equivbund}
\begin{enumerate}
\item $(A,u)|_{\tU_1}$ satisfies $\sig_\fix^*(A,u)=\mu(A,u)$ (viewing
  $\mu$ as a gauge transformation). 
\item By the above condition, $\sig_\fix^*(\tau(A,u))=\tau(A,u)$ on
  $\tU_1 \bs \{0\}$, so it descends to a gauged holomorphic map on
  $\tU_1 \bs \{0\}/\Z_\fix$. We require that this descended map is $(A,u)|_{U_2 \bs \{0\}}$.
\end{enumerate}

  \label{rem:holobundle}
  Holomorphic bundles on $\P(1,\fix)$ can be described in a similar
  way to the unitary bundles. For holomorphic bundles, the transition
  functions $\mu:\tU_1 \to G$ and $\tau:\tU_1\bs\{0\} \to G$ are
  holomorphic maps.  Now any holomorphic principal bundle over $\C$ is
  trivial (see \cite[Remark 19.6]{Cornalba}). So, the complex gauge
  equivalence class of a gauged holomorphic map can be specified by
  $u: \tU_1 \sqcup U_2 \to X$ and the transition functions $\tau$ and
  $\mu$.

\subsection{Standard form near infinity}

In this section we show that we
can complex gauge-transform a gauged holomorphic map on a weighted
projective line to a {\em standard form} near infinity.  We identify
the complement of the orbifold point
in $\P(1,n)$
 with $\C$, and let $B_R$ denote an open ball of radius $R$ around
 $0$.

\begin{proposition} \label{prop:connext}
Suppose $A$ is a connection on $P \to \P(1,\fix)$. There is a complex
gauge transformation $g \in \G(P)$, and a trivialization of $P$ over $\C$
so that $gA|_\C=\onD+\lambda d \theta$ on $\C \bs B_R$ for some $R>0$,
$\lambda \in \k$ satisfying $e^{2\pi \fix\lambda}=\Id$.  Conversely,
if $A$ is a connection on $\C$ such that $A=\onD + \lambda d\theta$ on $\C \bs B_R$, then it
extends to a connection on a principal bundle over $\P(1,\fix)$.
\end{proposition}

\begin{proof} Let $A$ be a
  connection on a principal bundle $P \to \P(1,n)$.
  We can transform $A$ to a flat connection in a neighborhood of
  infinity by complex gauge transformation as follows.  Choose $R_1>0$
  and consider $B_{2R_1}\subseteq \tU_1$. By Theorem \ref{thm:YMbdry},
  there is a unique $s: B_{2R_1}\to \k$ with $s|_{\partial
    B_{2R_1}}=0$, so that $e^{is}A$ is a flat connection.  By the
  uniqueness of $s$ and the symmetry of $A$ (see
  \eqref{eq:znsymmetry}), the complex gauge transformation $e^{is}$ is
  symmetric under the $\Z_n$ action, i.e.  $s \circ \sig_\fix =
  \Ad_\mu s$. Let $\eta:\tU_1 \to [0,1]$ be a radially symmetric
  cut-off function that is 1 on $B_{R_1}$ and vanishes on $\tU_1 \bs
  B_{2R_1}$. It is easy to see that $e^{i\eta s}$ defines a complex
  gauge transformation $g$ on all of $\P(1,\fix)$.  The connection
  $gA$ is flat near infinity.

Next we do a further unitary gauge transformation so that the
connection is in standard form, working over $U_2$. Let $R=
R_1^{-\fix}$. Choose a trivialization $P|_{U_2} \to U_2 \times G$
so that $gA$ is in radial gauge outside $B_R$, and since $gA$ is flat,
$gA=\onD+a(\theta) \d \theta$, for some $a:S^1 \to \k$. We now produce
a gauge transformation $k: U_2 \to K$ such that $kgA=\onD + \lambda
\d\theta$ for some $\lambda \in \k$. Let $k_1:[0,2\pi] \to K$ be the
solution of
\begin{align*}
\frac {k_1^{-1}dk_1} {\d \theta} &= a(\theta), & k_1(0)&=\Id.
\end{align*}
The path $\theta \mapsto k_1(\theta)$ can be homotoped to a geodesic
$\theta \mapsto e^{\lambda \theta}$, $\lambda \in \k$.  Then,
$e^{\lambda \theta}k_1^{-1}$ is a gauge transformation on $U_2 \bs B_R$
that is homotopic to the identity and it transforms $gA$ to
$\onD+\lambda \d \theta$ on $U_2 \bs B_R$. By using a cut-off function,
$e^{\lambda \theta}k^{-1}$ can be extended to a gauge transformation $k$
on all of $U_2$. The holonomy of $kgA$ about infinity is
$e^{2\pi \lambda}$. Since $g(A)$ has trivial holonomy for loops
close to $0$ in $\tU_1$, we have $e^{2\pi
  \fix\lambda}=\Id$.

For the converse, we construct a bundle $P \to \P(1,\fix)$. Set
$\mu=e^{-2\pi\lambda}$ and $\tau=e^{\fix\lambda \theta}$. We are given
$A|_{U_2}$. A connection $A|_{\tU_1\bs \{0\}}$ can be constructed using the transition
function $\tau$. The connection $A|_{\tU_1\bs \{0\}}$ is trivial on $B_{R^{-\fix}}\bs \{0\}
\subseteq \tU_1\bs \{0\}$, so it extends smoothly to a connection on $\tU_1$.
\end{proof}

\begin{remark}  \label{rem:orbibundleinv} 
\begin{enumerate} 
\item 
{\rm (Choice of orbifold singularity)} 
Suppose $A$ is a connection on the trivial bundle $\C \times K$ of the form mentioned in the
above proposition, i.e. $A=\onD+\lambda \d \theta$ on $\C \bs B_R$ with
$e^{2\pi \lambda \fix}=\Id$. We can extend $A$ to a connection on
principal bundles not just over $\P(1,\fix)$, but also $\P(1,m\fix)$
for any positive integer $m$.
\item \label {part:stdform} {\rm (Choice of standard form)} In the
  lemma above, the infinitesimal holonomy $\lambda$ produces the
  classifying path for the bundle as $\theta \mapsto e^{\lambda
    \theta}$. The choice of $\lambda$ is unique up to the action of
  $\Ad_K$. In our description of bundles over $\P(1,\fix)$, the homotopy class of this path
  can be recovered from the transition functions $\mu$, $\tau$ as the
  concatenation of the path $\theta \in [0,\frac {2\pi} n] \mapsto
  \tau^{-1}(r,0)\tau(r,\theta)$ and $t \in [-r,0] \mapsto
  \mu(-t,0)$. This is a continuous path because
  $\tau^{-1}(r,0)\tau(r,\frac {2\pi}n)=\mu(r,0)^{-1}$.
\end{enumerate} 
\end{remark}

The next result follows easily from Proposition \ref{prop:connext}.

\begin{proposition} \label{prop:ghext} {\rm (Standard form near infinity
for gauged holomorphic maps)} Let $P \to \P(1,\fix)$ be a principal
  $K$-bundle and let $(A,u)$ be a gauged holomorphic map from $P$ to
  $X$. There is a complex gauge transformation $g$ on $P$ and a
  trivialization of $P$ over $\C$ so that $g(A,u)$ satisfies the
  following:
\begin{enumerate}
\item There is a $\lambda \in \k$ so that $gA=\onD+\lambda \d \theta$ on
  $\C \bs B_R$ for some $R>0$.  The element $\lambda$ satisfies
  $e^{2\pi \fix\lambda}=\Id$.
\item There exists $x \in X$ such that for any $\theta \in [0,2\pi)$, $\lim_{r \to \infty}e^{-\lambda
    \theta}u(r,\theta)=x$ and $e^{2\pi \lambda}x=x$.
\end{enumerate}
Conversely, any gauged holomorphic map from $\C \times K$ to $X$ that
satisfies the above conditions for some $n$ extends to a map on $\P(1,\fix)$ for
some principal bundle $P \to \P(1,\fix)$. 
\end{proposition}

\begin{remark} {\rm (Topological invariants of affine vortices)}\label{rem:topinv} The second equivariant homology class of a gauged holomorphic map on
  $\P(1,n)$, $[(P,A,u)] \in H_2^K(X,\Q)$ is obtained by pushing
  forward the rational fundamental class of the domain $[\P(1,n)]$
  under a map $P \times_K X \to EK \times_K X$ given by a classifying
  map for $P$.  It is a deformation invariant of $(P,A,u)$.  By
  Theorem \ref{thm:main}, one can also associate to any finite energy
  affine vortex $(A,u)$ a homology class $[(P,A,u)]$. The class
  $[(P,A,u)] \in H_2^K(X,\Q)$ is integral if the $G$ action on
  $X^{\ss}$ is free, $n=1$. Consider the map
  $$f:H_2^K(X,\Q) \to H_2^K(\on{point},\Q)=H_2(BK,\Q) .$$ 
  The class $[P]:=f_*([(P,A,u)])$ determines the topology of the principal
  bundle $P \to \P(1,\fix)$. This topological information is precisely a
  choice of $\lambda \in \k$ satisfying $e^{2\pi \fix \lambda}=\Id$
  modulo $\Ad_K$ (see Remark \ref{rem:orbibundleinv}).
\end{remark}

\subsection{Equivariant convexity}\label{subsec:equivconvex}
The notion of convexity for symplectic manifolds first arose in work
by Eliashberg and Gromov \cite{Eliashberg:convexity}. Convex
symplectic manifolds have similar properties to compact
manifolds. This idea is extended to Hamiltonian symplectic manifolds
in Cieliebak et al. \cite{CGMS}. An important example of equivariantly
convex spaces are symplectic vector spaces with a linear group action and a
proper moment map.

\begin{definition}\label{def:convexinfty}
A K\"ahler manifold $(X,\om,J)$ with a Hamiltonian $K$-action is {\em
  equivariantly convex at infinity} if there is a $K$-invariant proper
function $f:X \to \R_{\geq 0}$ and a value $c_0$ such that if $f(x)>c_0$, then
\begin{align}\label{eq:eqconv}
\lan \nabla_\xi \nabla f, \xi \ran \geq 0 \quad \forall \xi \in T_xX, \quad df(J\Phi(x)_X) &\geq 0.
\end{align}
Here $\nabla f \in \on{Vect}(X)$ is the gradient
vector field of $f$ with respect to the metric $\om(\cdot,J\cdot)$. 
\end{definition}
The above definition is equivalent to the definition in \cite{CGMS}, where there is an additional term $\lan \nabla_{J\xi} \nabla f, J\xi \ran$ in the left hand side of the first equation above. But, when $X$ is K\"ahler $\nabla_{J\xi}=J\nabla \xi$ so that term is equal to $\lan \nabla_\xi \nabla f, \xi \ran$.
 The condition \eqref{eq:eqconv} implies that the image of a finite energy
affine vortex with bounded image is contained in the compact set
$\{f \leq c_0\}$ - see Proposition 11.1 in \cite{GS}.

\subsection{Sobolev Spaces}
In this paper, most analytic proofs are carried out in Sobolev
completions of the spaces described above. We set down our conventions
here. Let $\Sig$ be a compact Riemannian manifold, possibly with a smooth boundary and $E
\to \Sig$ a vector bundle. Fix a smooth covariant derivative
$\nabla$ on $E$.
For any $k \in \Z_{\geq 0}$, $p \in
\R_{>1}$ and a smooth section $\sig \in\Gamma(\Sig,E)$, define a norm
$$\Mod{\sig}_{W^{k,p}}^\nabla:=\bigl(\sum_{i=0}^k\int_\Sig|\nabla^i\sig|^p\dvol_\Sig\bigr)^{1/p}.$$
The norms for different choices of $\nabla$ are equivalent, so
$\nabla$ is dropped from the notation. The space $W^{k,p}(\Sig,E)$ is
the completion of the space of smooth sections under this norm. If $\Sig \subset \R^n$ is an open subset whose boundary $\ol \Sig \bs \Sig$ is smooth in $\R^n$, then the same definition carries over, i.e. $W^{k,p}(\Sig,E):=W^{k,p}(\ol \Sig,E)$. Using
such norms one can define the Sobolev completion $\A^{k,p}(P)$ of the
space of connections on a principal bundle, when $(k+1)p>\dim
\Sig$. We can also define Sobolev spaces of maps between manifolds
$W^{k,p}(\Sig,X)$ when $kp>\dim\Sig$. We refer to Appendix B of the
book \cite{Weh:Uh} for details. We use the notation
$H^k:=W^{k,2}$ for any $k \in \Z$.

Sobolev spaces can be defined when the value of the exponent $k$ is
negative. For $k\geq 0$, let $W^{k,p}_0(\Sig,E)$ be the completion of
the space of compactly supported smooth sections $C^\infty_0(\Sig,E)$
under the $W^{k,p}$ norm. Also, let $p':=\frac p
{p-1}$ and assume that the
bundle $E$ has a Riemannian metric $g$. Then,
$$\Mod{\sig}_{W^{-k,p}}:=\sup_{\sig' \in W_0^{k,p'}(\Sig,E)}\int (\sig,\sig')_g\dvol_\Sig.$$
Then, $W^{-k,p}(\Sig,E)$ is the completion of smooth sections
under the dual norm of $W^{k,p'}_0$. It is the dual of the space
$W^{k,p'}_0(\Sig,E)$ under the $L^2$-pairing of sections.

The operator
$$\d:W^{-k,p}(\Sig,P(\k)) \to W^{-k-1,p}(\Om^1(\Sig,P(\k)))$$
is bounded because it is obtained by dualizing
$$\d^*:W^{k+1,p}_0(\Om^1(\Sig,P(\k))) \to W^{k,p}_0(\Sig,P(\k)).$$
Combining this observation with the Sobolev multiplication Theorem (Proposition \ref{prop:sobmult}) we
see that for $p>\dim \Sig$, the curvature of an
$L^p$ connection is a distributional two-form in $W^{-1,p}$. Using these weak connections, a weak version of vortices can be defined.  Let $\Sig$ be a compact Riemann surface (possibly with a smooth boundary) and $(A,u)$ lie in $
(L^p \times W^{1,p})(\Sig)$ for some $p>2$. For such a pair, the
vortex equation holds in a weak sense if $F_{A,u}$ vanishes in
$W^{-1,p}$, i.e. for all $\xi \in W^{1,p}_0(\Sig,P(\k))$, we have
$\int_\Sig \lan F_{A,u},\xi\ran=0$. If $(A,u)$ is a vortex weakly, the energy of $(A,u)$ is well-defined because the curvature
 $F_A$ is equal to  $-\Phi(u) \dvol_\Sig$ which is in $C^0$.

 \section{From holomorphic maps to vortices}\label{sec:outside}

\subsection{Statement of main theorem}
 We prove a slightly refined version of the main result Theorem
 \ref{thm:premain} relating affine vortices to gauged holomorphic maps
 from orbifold lines.  This requires the following analytic
 definition: Fix $p>2$. We call a gauged holomorphic map $(A,u)$ on $P
 \to \P(1,n)$ {\em $p$-bounded} if it is smooth on $\C$ and on $B_{\tilde
   R}\subset \tU_1$, which is a neighborhood of $\infty$,
 $(A,u)|_{B_{\tilde R}} \in L^p \times W^{1,p}$. A {\em (complex)
   gauge transformation} on $P$ is {\em $p$-bounded} if it
 is smooth on $\C$ and in $W^{1,p}$ in a neighborhood of $\infty$.
 Denote by $\G(P)_{\on{bd}}$ the group of $p$-bounded
 complex gauge transformations. By the Sobolev embedding theorem, any
 $p$-bounded $(A,u)$ resp. $g$ is continuous and so
 $u(\infty)$ resp. $g(\infty)$ is well-defined.  The following is the
 analytic version of the main result of the paper.

\begin{theorem}\label{thm:main} 
  Suppose $X$ is a K\"ahler manifold with Hamiltonian action of a
  compact Lie group $K$, and $X$ is either compact or equivariantly convex
  at infinity with a proper moment map. Let $G$ be the
  complexification of $K$, and suppose $G$ action on $X^{\ss}$ has finite stabilizers. Let $n$ be an integer such that for any $x \in X^{\ss}$,
  $|G_x|$ divides $n$. Fix $2<p<2(1+\frac 1 n)$.
  \begin{enumerate}
  \item \label{part:mainthm1} Let $(A,u)$ be a gauged holomorphic map
    from a principal bundle $P \to \P(1,\fix)$ that satisfies
    $u(\infty) \in X^{\ss}$.  There is a $p$-bounded complex
    gauge transformation $g \in \G(P)_{\on{bd}}$ such that
    $g(A,u)|_\C$ is a smooth finite energy symplectic
    vortex with bounded image. Furthermore, $g$ is unique up to left multiplication by a
    $p$-bounded unitary gauge transformation.

  \item \label{part:mainthm2} Conversely, given any finite
    energy symplectic vortex with bounded image, there is a unique (up to isomorphism)
    $K$-bundle $P \to \P(1,n)$ so that $(A,u)$ extends to a $p$-bounded
    gauged holomorphic map on $P$. There is a $p$-bounded
    complex gauge transformation $g \in \G(P)_{\on{bd}}$ so that
    $g(A,u)$ is smooth over $\P(1,n)$. The gauged holomorphic map
    $g(A,u)$ is unique up to complex gauge transformations in $\G(P)$.
  \end{enumerate}
\end{theorem}
\begin{remark}{\rm (Relating the statements of Theorem \ref{thm:premain}
  and Theorem \ref{thm:main})} With our new notations, the statement of
  Theorem \ref{thm:premain} can be completed as follows. The section
  $\sig:\P(1,n) \to P_\C/K$ is smooth on $\C \subset \P(1,n)$ and in
  $W^{1,p}$ in a neighborhood of $\infty$. 

We now explain how Theorem \ref{thm:premain} follows from Theorem \ref{thm:main}. Given a gauged holomorphic
  map $(P_\C,u)$, choose any smooth section $\sig:\P(1,n) \to
  P_\C/K$. The choice of $\sig$ is equivalent to the choice of a
  $K$-bundle $\iota: P \hra P_\C$. The resulting unitary object
  $(A,u)$ is a gauged holomorphic map defined on $P$. Let $g\in
  \G(P)_{\on{bd}}$ be the complex gauge transformation produced by Part
  \eqref{part:mainthm1} of Theorem \ref{thm:main}. Then, the inclusion
  $g \circ \iota: P \hra P_\C$ corresponds to a section of $P_\C/K$
  that satisfies the conclusions of Theorem \ref{thm:premain}.
\end{remark}
In this section, we prove the first part of Theorem \ref{thm:main}. The second
part is proved in Section \ref{sec:last}. 

\subsection{The heat flow}

In this subsection, we state some relevant results about heat flow
from the work of the first author \cite{Venu}. Let
On closed surfaces a result of Mundet \cite{Mund} produces a vortex
from a gauged holomorphic map satisfying a stability condition.  We
will use an alternative method of producing this gauge transformation
using the heat flow from the work of the first author \cite{Venu}.
Let $c_0$ denote the minimal norm of the moment map on the unstable
locus,
\begin{equation}\label{eq:c0def}
  c_0=\inf\{|\Phi(x)|: x \in X, |K_x|=\infty\}.
\end{equation}
By assumption \ref{ass:freeaction}, $c_0>0$.

\begin{theorem}\label{thm:hk} {\rm(Heat flow for gauged holomorphic maps \cite[Theorem 4.4.1]{Venu}) }
  Let $\Sig$ be a compact Riemann surface without boundary and
  $(A_0,u_0)$ is a gauged holomorphic map on a principal bundle $P \to
  \Sig$. In addition, assume that $(A_0,u_0)$ satisfies $E(A_0,u_0)
  \leq c_0^2\vol(\Sig)$. Then, there is a complex gauge transformation
  $g=e^{i\xi}$, $\xi \in W^{2,p}(\Sig,P(\k))$ (for any $p>2$) such
  that $g(A,u)$ is a vortex. Up to unitary gauge equivalence, $g(A,u)$
  is the unique vortex in the complex gauge orbit of $(A,u)$.
\end{theorem}

This result is obtained by studying the gradient flow of $(A_0,u_0)$
in the space of gauged pairs on $\Sig$ under the vortex functional
$$(A,u) \mapsto \Mod{F_{A,u}}_{L^2(\Sig)}^2, \quad F_{A,u}:=*F_A+\Phi(u).$$
The trajectory of the flow $t \mapsto (A_t,u_t)$ lies in the complex
gauge orbit of $(A_0,u_0)$ and it converges to a critical point of the
functional. The bound on the energy of $(A_0,u_0)$ ensures that the
critical point corresponds to $F=0$, i.e. it is a vortex and that the
limit is also in the complex gauge orbit of $(A_0,u_0)$.  The flow has
the property of decreasing energy. This is seen by the following
energy identity.

\begin{lemma} {\rm (Cieliebak et
    al. \cite{CGS})} \label{lem:energytop} Let $\Sig$ be a 
  compact Riemann surface without boundary and $P$ a principal $K$-bundle on it. A pair
  $(A,u) \in \A(P) \times \Gamma(\Sig,P(X))$ satisfies
  \begin{equation}\label{eq:vortenergy}
    \begin{split}
      \hh\int_\Sig |F(A)|^2&+|\Phi \circ u|^2 + |\d_A u|^2 \dvol_\Sig \\
      &=\int_\Sig |\delbar_A u|^2 +\hh|*F_A + \Phi(u)|^2 \dvol_\Sig +
      \lan \om_X-\Phi,u\ran,
    \end{split}
  \end{equation}
  where $\lan \om_X-\Phi,u\ran = \int_\Sig u^*\om - d\lan
  \Phi(u),A\ran$.
\end{lemma}

\noindent The pairing
$\lan \om_X-\Phi,u\ran$ is a topological invariant of $(A,u)$ and is
preserved by the flow. Since $\Mod{F_{(A_t,u_t)}}_{L^2}^2$ decreases
with time $t$, the same is the case with the energy $E(A_t,u_t)$.

We recall several other results 
about gauged holomorphic maps on Riemann surfaces with boundary needed
for the proof of Theorem \ref{thm:main} \eqref{part:mainthm1}. The
first one says that any gauged holomorphic map which is an approximate
solution to the vortex equations may be complex gauge transformed into
one.

\begin{proposition}\label{prop:orbitvortex}{\rm(\cite[Proposition 4.3.1]{Venu})} Let $p > 2$. Suppose
  $\Sig$ is a compact connected Riemann surface with a smooth 
  boundary. Let $(A_i,u_i) \in \H(P,X)_{L^p \times W^{1,p}}$ be a
  sequence and $(A_\infty,u_\infty) \in \H(P,X)_{L^p \times W^{1,p}}$
  be such that $A_i \to A_\infty$ in $L^p$ and there is a finite set
  $Z \subseteq \Sig$ so that $u_i \to u_\infty$ in $C^0$ on compact
  subsets of $\Sig \bs (Z\cup \partial \Sig)$. Also,
  $F_i:=*F(A_i)+u_i^*\Phi \to 0$ in $W^{-1,p}$. Then, there exist constants $C$ and $i_0$
  so that for $i>i_0$, there is a complex gauge transformation $\exp
  i\xi_i$, $\xi_i \in W^{1,p}_0(\Sig,P(\k))$ so that $(\exp
  i\xi_i)(A_i,u_i)$ is a vortex and satisfies
  $\Mod{\xi_i}_{W^{1,p}}<8C\Mod{F_i}_{W^{-1,p}}$.
\end{proposition}
The next result we will need is Proposition 4.3.3 in
\cite{Venu}. Roughly it says that in a complex gauge orbit, there is
at most one vortex up to gauge. The proof is reproduced, because it
will be useful in understanding the corresponding result for affine
vortices.
\begin{proposition} \label{prop:atmost1vortex} Let $p>2$ and $\Sig$ be a compact connected 
  Riemann surface with a smooth boundary. Let $(A_0,u_0)$,
  $(A_1,u_1) \in L^p \times W^{1,p}$ be vortices on a principal bundle
  $P \to \Sig$ that are related by a complex gauge transformation $g
  \in \G^{1,p}(P)$, i.e. $(A_1,u_1)=g(A_0,u_0)$ and assume $g(\partial
  \Sig) \subseteq K$. Then, $(A_0,u_0)$ and
  $(A_1,u_1)$ are gauge-equivalent, i.e. $g \in \K^{1,p}(P)$.
\end{proposition}
\begin{proof}
  After a gauge transformation, we can assume
  $(A_1,u_1)=e^{i\xi}(A_0,u_0)$, where $\xi \in
  \Gamma(\Sig,P(\k))_{1,p}$ and $\xi|_{\partial \Sig}=0$.  Let
  $(A_t,u_t):=e^{it\xi}(A_0,u_0)$. We know $F_{A_0,u_0}=F_{A_1,u_1}=0$
  weakly. For $\xi|_{\partial \Sig}=0$,
  \begin{align*}
    \ddt \int_\Sig \lan *F_{A_t,u_t},\xi \ran &= \int_\Sig\langle \d_{A_t}^*\d_{A_t}\xi+u_t^*d\Phi(J(\xi)_{u_t}),\xi\rangle_\k\\
    &=\Mod{\d_{A_t}\xi}_{L^2}^2 +
    \int_\Sig\omega_{u_t}((\xi)_{u_t},J(\xi)_{u_t})\geq 0.
  \end{align*}
  The inequality is strict for non-zero $\xi$. So, $\xi=0$ and
  $(A_0,u_0)$ and $(A_1,u_1)$ are gauge-equivalent.
\end{proof}

\subsection{Elliptic regularity for gauged holomorphic maps}

This section contains some results related to elliptic regularity for
vortices and gauged holomorphic maps required in the proof of Theorem
\ref{thm:main} \eqref{part:mainthm1}. Most of these lemmas are similar
to results already present in the literature. We include them here,
because we require extensions to slightly lower regularity spaces.
\begin{lemma}{\rm (Gromov convergence for vortices \cite{Ott},
  \cite{Zilt:QK})}\label{lem:vortconv} Let
  $p>2$ and $\Sig_i$ be a sequence of precompact sets exhausting a
  Riemann surface $\Sig$,
  \begin{align*}
    \Sig_1 \Subset \Sig_2 \Subset \dots \Subset \Sig,&
    &\bigcup_i\Sig_i=\Sig.
  \end{align*}
  Suppose for each $i$, $(A_i,u_i)$ is a smooth vortex on $\Sig_i$,
  $$\sup_iE(\Sig_i,(A_i,u_i))<\infty$$
  and there is a compact set $S \subset X$ containing the images of
  all the $u_i$. Let $Z$ denote the set of points $z$ for which there is a
  sequence $z_i \to z$ in $\Sig$ such that $|\d_{A_i}u_i(z_i)| \to
  \infty$ as $i \to \infty$. Then, the set $Z$ is finite and after passing to a subsequence, there are gauge
  transformations $k_i \in H^2(\Sig_i,K)$ and a finite energy vortex $(A_\infty,u_\infty) \in H^1_{\loc}
  \times H^2_{\loc}$ over $\C$ such that
  \begin{enumerate}
  \item $k_iA_i \weakto A_\infty$ in $H^1$, and strongly in $L^p$, on
    compact subsets of $\Sig$.
  \item $u_i \weakto u_\infty$ in $H^2$, and strongly in $W^{1,p}$ and
    $C^0$, on compact subsets of $\Sig \bs Z$.
  \end{enumerate}
\end{lemma}
\begin{proof} This lemma is a combination of results in \cite{Zilt:QK}
  and \cite{Ott}. We provide an outline of the proof. The bounded
  energy condition implies a curvature bound
  $$\Mod{F(A_i)}_{L^2(\Sig_i)}<c$$
  for all $i$. By Uhlenbeck's theorem for non-compact domains (Theorem
  A' in \cite{Weh:Uh}), after passing to a subsequence, there are
  gauge transformations $k_i \in H^2(\Sig_i)$ and a limit connection
  $A_\infty$ on the trivial bundle $\Sig \times K$ such that $k_iA_i
  \weakto A_\infty$ in $H^1$ on compact subsets of $\Sig$.
  
  We first prove convergence on compact sets for bounded first
  derivatives.  On a compact set $Q \subset \Sig$, if there is a bound
  on $\Mod{\d_{A_i}u_i}_{L^p(Q)}$, then a subsequence of $k_iu_i$
  converges weakly in $H^2(Q)$.  To see this, write
  $k_iA_i=\onD+a_i$. Since all the $u_i$ map to a compact set and
  $a_i$ is bounded in $L^p$, $(a_i)_{u_i}$ is also bounded in $L^p$,
  hence there is an $L^p(Q)$ bound on $d(k_i u_i)$. This implies $k_i
  u_i$ converges weakly in $W^{1,p}(Q)$ and strongly in $C^0$ to a
  limit $u_\infty$ (see Theorem B.4.2 in \cite{MS}). By Lemma
  \ref{lem:ghellipreg} below, after passing to a subsequence $k_iu_i
  \weakto u_\infty$ in $H^2(Q)$. A standard diagonalization argument
  shows that there is a subsequence that works for all compact subsets
  $Q$.

  In the absence of a first derivative bound, one has bubbling.  As in
  Ott \cite{Ott}, there is a finite set $Z \subset \Sigma$ where
  bubbling occurs. That is, after passing to a subsequence,
  $\Mod{\d_{A_i}u_i}_{L^p(Q)}$ is bounded on compact subsets $Q \subset
  \Sig \bs Z$. There is a map $u_\infty:\Sig \bs Z \to
  X$ such that $k_iu_i$ converges weakly to $u_\infty$ in $H^2(Q)$ for
  all compact subsets $Q \subset \Sig \bs Z$. The map $u_\infty$
  extends to a continuous map on $\Sig$.  Since we are in the K\"ahler
  case, this has a much easier proof than that in Ott \cite{Ott}:
  Choose $z_0 \in Z$ and a small neighborhood $U \subset \Sig$ so that
  $U \cap Z=\{z_0\}$.  By Lemma \ref{lem:toflat}, there is a complex
  gauge transformation $g \in H^2(U) \hra C^0(U)$ such that
  $gA_\infty$ is the trivial connection. Complex gauge transformations
  preserve holomorphicity, so $\delbar(gu_\infty)=0$ on $U \bs
  \{z_0\}$, $gu_\infty$ is smooth on $U\bs \{z_0\}$ and the image
  $gu_\infty(U \bs \{z_0\})$ is contained in a compact set. By the
  removable singularity theorem of complex analysis (\cite[Theorem 3.1]{Stein_Shakarchi}), we have $gu_\infty$
  extends smoothly over $z_0$. Therefore $u_\infty$ extends over $z_0$ as an $H^2$ map, and so $u_\infty \in H^2_{\loc}$.
\end{proof}
The following elliptic regularity result (Theorem 3.1 in Cieliebak et
al.\cite{CGMS}) is standard, but we provide a self-contained proof.
\begin{lemma}\label{lem:ghellipreg} {\rm (Elliptic regularity for
    gauged holomorphic maps)}
  Let $s \geq 1$ be an integer, $p>2$, $\Sig \subseteq \C$ be a pre-compact open set with smooth boundary and $(A_i,u_i)$ be a sequence of
  gauged holomorphic maps on $\Sig$ such that $A_i \weakto A_\infty$
  in $H^s(\Sig)$ and $u_i \weakto u_\infty$ in $W^{1,p}(\Sig)$, then
  $u_i \weakto u_\infty$ in $H^{s+1}(\Sig')$, for any compact set
  $\Sig'$ that is contained in $\on{int}(\Sig)$.
\end{lemma}

\begin{proof} The proof is by induction on $s$. We first assume the
  result for some $s > 1$ and prove it for $s+1$. Note that it
  suffices to work locally in $X$: Choose an atlas $X=\cup_\alpha
  \V_\alpha$ such that $\V_\alpha$ is bi-holomorphic to an open subset
  of $\C^N$. Since $u_i \to u_\infty$ in $C^0$, we can find a finite
  cover $\Sig=\cup_\beta \U_\beta$ such that for $u_i(\U_\beta)$ is
  contained in a single $\V_\alpha$ for large $i$. So, now we can
  think of $u_i$ as mapping to $\C^N$.

  In each chart we apply a combination of Sobolev multiplication and
  regularity theorems.  As in \cite{CGMS}, write $A_i=\onD + \Theta_i
  dx + \Psi_i dy$, where $\Psi_i$, $\Theta_i \in H^s(\Sig,\k)$ and the
  holomorphicity equation for $(A_i,u_i)$ is
  \begin{align}\label{eq:holo}
    -\delbar u_i = (\Theta_i)_{u_i}+J_X(\Psi_i)_{u_i}.
  \end{align}
  See \eqref{eq:dau} for the notation $(\Theta_i)_{u_i}$. We know that
  both $A_i$ and $u_i$ weakly converge in $H^s$, so $\Psi_i$,
  $\Theta_i$ and $u_i$ are uniformly bounded in $H^s$. We next show
  that $(\Theta_i)_{u_i}$ and $(\Psi_i)_{u_i}$ are uniformly bounded
  in $H^s$. For this, define an operator $L_x$ for every $x \in X$,
  \begin{equation}\label{eq:defL}
    L_x : \k \to T_xX, \quad  \xi \mapsto \xi_X(x).
  \end{equation}
  The function $L:X \to \Hom(\k, TX) \simeq \Hom(\k, \C^N)$ given by
  $x \mapsto L_x$ is smooth. So, given $u:\Sig \to X$, we obtain a
  section
  $$L(u):=L \circ u \in \Gamma(\Sig, \Hom(\k,u^*TX)) \simeq
  \Gamma(\Sig, \Hom(\k,\C^N)).$$
  The term $(\Theta_i)_{u_i}$ can be seen as a product
  $(\Theta_i)_{u_i}=L(u_i)\Theta_i$. In the following discussion, the
  constant $c$ indicates a constant that is independent of $i$, whose
  value varies across appearances. Since $L$ is smooth,
  $\Mod{L(u_i)}_{H^s(\U_\alpha)}<c$ for all $i$, $\alpha$. By Sobolev
  multiplication (Proposition \ref{prop:sobmult}), for $s > 1$,
  \begin{equation}\label{eq:aumult}
    \Mod{L(u_i)\Theta_i}_{H^s(\U_\alpha)} \leq c\Mod{L(u_i)}_{H^s(\U_\alpha)} \Mod{\Theta_i}_{H^s(\U_\alpha)}.
  \end{equation}
  By holomorphicity \eqref{eq:holo}, $\Mod{\delbar
    u_i}_{H^s(\U_\alpha)} < c$ for all $i$, $\alpha$. By elliptic
  regularity for curves in $\C^N$,
  $$\Mod{u_i}_{H^{s+1}(\U'_\alpha)} \leq c(\Mod{\delbar u_i}_{H^s(\U_\alpha)} + \Mod{u_i}_{L^2(\U_\alpha)}).$$
  where $\ol \U_\alpha' \subseteq \U_\alpha$ and $c$ depends on
  $\U_\alpha$, $\U_\alpha'$. By choosing $\U_\alpha'$ such that they
  cover $\Sig'$, we obtain a uniform bound on
  $\Mod{u_i}_{H^{s+1}(\Sig')} $ and so, after passing to a
  subsequence, $u_i \weakto u_\infty$ in $H^{s+1}(\Sig')$.

  It remains to prove the result for $s=1$. The Sobolev multiplication
  step \eqref{eq:aumult} is not applicable for $s=1$. Instead, we now
  have a bound $\Mod{L(u_i)}_{W^{1,p}(\U_\alpha)}<c$ and then by the
  Sobolev multiplication theorem (Proposition \ref{prop:sobmult}),
  \begin{equation*}
    \Mod{L(u_i)\Theta_i}_{H^1(\U_\alpha)} \leq c\Mod{L(u_i)}_{W^{1,p}(\U_\alpha)} \Mod{\Theta_i}_{H^1(\U_\alpha)}.
  \end{equation*}
  The rest of the proof of Lemma \ref{lem:ghellipreg} is same as the
  $s>1$ case.
\end{proof}

\begin{lemma}{\rm (Regularity for vortices)} \label{lem:vortexreg} Let
  $p>2$. Given a finite energy vortex $(A,u) \in (L^p_{\loc} \times
  C^0)(\C)$ on the trivial bundle $\C \times K$, there is a gauge
  transformation $k \in W^{1,p}_{\loc}(\C)$ such that $k(A,u)$ is smooth
  on $\C$.
\end{lemma}
\begin{proof}
  We first carry out the proof for the closure $\Sig$ of a bounded open subset with smooth boundary in $\C$. The proof is by induction. We first deal with the case when $(A,u)$
  has high enough regularity. We assume $(A,u) \in (W^{m,p} \times
  W^{m+1,p})(\Sig)$, where $mp>2$. The idea of the proof is to find a
  smooth reference connection $A_0$ and put $A$ is Coulomb gauge with
  respect to $A_0$ by a gauge transformation $k \in
  W^{m+1,p}$. Writing $kA=A_0+\alpha$, we have
  $\d_{A_0}^*\alpha=0$. Since $u \in W^{m+1,p}(\Sig)$, $\Phi(ku)$ is
  also in the same class and by the vortex equation $F_{kA} \in
  W^{m+1,p}(\Sig) \hra W^{m,p}(\Sig)$. Writing
  \begin{equation}\label{eq:dareg}
    \d_{A_0}\alpha=F_{kA} - \hh [\alpha \wedge \alpha] \implies
    \d_{A_0}\alpha \in W^{m,p}.
  \end{equation}
  So, we have $\alpha \in W^{m+1,p}_{\on{loc}}(\Sig \bs \partial
  \Sig)$ and so, $kA \in W^{m+1,p}_{\loc}(\Sig \bs \partial \Sig)$. By
  elliptic regularity (similar to proof of Lemma \ref{lem:ghellipreg},
  see \eqref{eq:holo}), we have $W^{m+2,p}_{\loc}$ control over $ku$
  on $\Sig \bs \partial \Sig$. This procedure can be applied
  inductively to gauge transform $(A,u)$ to a smooth vortex on $\Sig
  \bs \partial \Sig$.

  Now, we discuss the case when $(A,u)$ has lower regularity than what
  is required in the previous paragraph, i.e. when $A \in
  L^p_{\on{loc}}$. As earlier, it is possible to gauge transform $A$
  to a connection in Coulomb gauge (see \cite[Theorem 8.3]{Weh:Uh})
  with respect to a nearby connection $A_0$, i.e. there exists $k \in
  W^{1,p}(\Sig,K)$ such that $kA=A_0+\alpha$ and
  $\d_{A_0}^*\alpha=0$. However, the next step \eqref{eq:dareg}
  requires that $mp>2$, which fails when $(A,u) \in L^p_{\loc} \times
  C^0$. The regularity of $\alpha$ is brought up by a bootstrapping
  procedure using the equation in \eqref{eq:dareg}. First assume
  $2<p<4$. Set $q_0:=p$. There exists $m$ and a sequence
  $2<q_0<\dots<q_{m-1}\leq 4$, $q_m>4$ such that
  $$q_0=p,\quad q_i<\frac {2q_{i-1}} {4-q_{i-1}} \quad i\geq 1.$$
  Suppose $\alpha \in L^{q_i}$. Then, by H\"older's inequality
  $[\alpha \wedge \alpha] \in L^{q_i/2}$. Since $ku$ is in $C^0$, the
  term $\Phi(ku)$ is also in $C^0$, and by the vortex equation
  $F_{kA}$ is in $C^0$, hence also in $L^{q_i/2}$. Bootstrapping using
  \eqref{eq:dareg}, we get $\alpha \in W^{1,q_i/2} \hra
  L^{q_{i+1}}$. By repeating this procedure, we end up with $\alpha
  \in L^{q_m}$. If $p>4$, then $m$ is equal to $0$ in the above
  discussion, and we straightaway have $\alpha \in L^{q_m}$.  Using
  \eqref{eq:dareg} one more time, we get $\alpha \in W^{1,q_m/2}$. By
  elliptic regularity (similar to proof of Lemma \ref{lem:ghellipreg},
  see \eqref{eq:holo}), $u \in W^{2,q_m/2}_{\loc}(\Sig \bs \partial
  \Sig).$ Now, the inductive step of the previous paragraph is
  applicable.

The proof extends to all of the complex line via an exhaustion argument carried out in Proposition 77 in \cite{Zilt:QK}, which carries over to this case with lower regularity.
\end{proof}

\subsection{Proof of Theorem \ref{thm:main}\eqref{part:mainthm1}}

We now present the proof of the main result Theorem
\ref{thm:main}\eqref{part:mainthm1} modulo some technical results
which are 
deferred to the following subsections.

\begin{proof} {\rm(Proof of Theorem \ref{thm:main} \eqref{part:mainthm1})}
Starting from a stable gauged holomorphic map, we first produce a
sequence of vortices on exhausting subsets of the complex line, which
has a limit modulo bubbling that is a finite energy affine vortex. Given
a stable gauged holomorphic map $(A,u)$ defined on a principal bundle $P$
over $\P(1,n)$ (resp. $\P^1$), we may assume that $u(\infty) \in
\Phi^{-1}(0)$. Since $u(\infty) \in X^\ss$, this condition can be
achieved by a complex gauge transformation. This preliminary step is
to simplify arguments later in the proof. We now apply Proposition
\ref{prop:seq1} below to
obtain a sequence of vortices $(A_i,u_i)$ defined on the annuli
$A(1/R_i,R_i)$ (resp. balls $B_{R_i}$) that satisfy
$(A_i,u_i)=g_i(A,u)$ where $g_i \in H^2(B_{R_i})$ is a complex gauge transformation. The sequence
converges in the sense of Proposition \ref{prop:seq1} 
 to a finite energy affine
vortex $(A_\infty,u_\infty)$. We remark that from the output of
Proposition \ref{prop:seq1}, we replace $g_ik_i$ by $g_i$ and drop
$k_i$ from our notation in this proof.

  In the remainder of the proof, we will show that the limit vortex is complex gauge equivalent to
  every vortex in the sequence. We first construct the
  necessary complex gauge transformation outside a large ball.  Since
  $(A_\infty,u_\infty)$ has finite energy and bounded image, by Proposition \ref{prop:infvortsing}, $\Phi(u_\infty(z))\to 0$ as
  $z\to \infty$. 
So we can choose an $R>0$ such that $u_\infty(\C \bs B_R) \subseteq
X^{\ss}$ and $Z \subseteq B_R$. After possibly passing to a
subsequence, we have $u_i(\C \bs B_R) \subseteq X^{\ss}$.  
Using Lemma \ref{lem:transfconv}, after passing to a subsequence
again, $g_i$ has a limit $g_\infty \in W^{1,p}_{\on{loc}}$.
In particular, we have
$$g_i \to g_\infty \quad \text{ in $W^{1,p}(Q,G)$ for all compact subsets $Q \subset \C \bs B_R$}.$$

We remark that in the case that the $G$-action on $X^\ss$ is free, the proof of this statement is more
straightforward and does not require Lemma \ref{lem:transfconv}:
Observe that by Assumption \ref{ass:freeaction}, semistable orbits are
relatively closed in $X^\ss$. This implies that for all $x \in \C \bs
B_R$, the points $u_i(x)$ and $u_\infty(x)$ are in the same
$G$-orbit. By the free action, there is a unique $g_\infty$ as
above. 
Further the $W^{1,p}$ convergence of the maps $u_i$ in compact subsets of $\C \bs B_R$ implies a similar convergence of the sequence $g_i$ to $g_\infty$. 
By Lemma \ref{lem:cgaugea}, we have
\begin{equation}\label{eq:glim}
  (A_\infty,u_\infty)=g_\infty(A,u) \quad \text{on $\C \bs B_R$.}
\end{equation}

The limit gauge transformation produced outside the ball
extends to inside the ball. By Proposition \ref{prop:extend_inball},
the gauge transformation $g_\infty$ extends to all of $\C$ while satisfying
$(A_\infty,u_\infty)=g_\infty(A,u)$ on $\C$. The complex gauge transformation 
$g_\infty$ is in $W^{1,p}_{\on{loc}}(\C)$.

We emphasize that the above argument shows there is no bubbling, i.e.
$Z=\emptyset$ and in the orbifold case $0$ is not a singular point. By the conclusion of Proposition \ref{prop:extend_inball}, the sequence $(A_i,u_i)$ which is equal to $g_i(A,u)$ converges to $g_\infty(A,u)$ weakly in $(L^p \times W^{1,p})(B_R)$, and this limit is same as $(A_\infty,u_\infty)$.
 So, the quantity $\Mod{\d_{A_i}u_i}_{L^p(B_R)}$, which is gauge
invariant, is uniformly bounded for all $i$. By Proposition 78 in
\cite{Zilt:QK}, modulo gauge transformations, the sequence $(A_i,u_i)$
converges smoothly to a limit vortex on compact subsets of
$\on{int}(B_R)$. This implies that $|d_{A_i}u_i|_{L^\infty}$ is
uniformly bounded on a small compact neighborhood of $Z$, which contradicts the characterization of $Z$ in Proposition \ref{prop:seq1}. So there is no bubbling, 
that is, the bubbling set $Z$ is empty.

We next show that the complex gauge transformation $g_\infty$ needed
to make the pair into a vortex has the claimed regularity, that is,
$g_\infty$ is $p$-bounded on the bundle $P$ over the weighted projective line.  Using
Lemma \ref{lem:vortexreg}, we can modify $g_\infty$ by a gauge
transformation in $W^{1,p}_{\loc}$, so that $g_\infty(A,u)$ is
smooth. By applying Lemma \ref{lem:complexgreg} in neighborhoods in
$\C$, we conclude that $g_\infty:\C \to G$ is smooth. Write
$g_\infty=k_\infty e^{i\xi_\infty}$. In the beginning of the proof, we assumed $\Phi(u(\infty))=0$. We also have $\lim_{r \to
  \infty}u_\infty(r,\theta) \in \Phi^{-1}(0)$ for any $\theta$. This
implies that $\xi_\infty(r,\theta) \to 0$ as $r \to \infty$. The complex gauge transformation $e^{i\xi_\infty}$ on $P$ is
smooth on $\C$ and continuous at $\infty$. We now show that
$e^{i\xi_\infty}\in \G(P)_{\on{bd}}$. Given the finite energy vortex
$e^{i\xi_\infty}(A,u)$, Theorem \ref{thm:main}\eqref{part:mainthm2}
shows that the vortex extends to a gauged holomorphic map over a
principal $K$-bundle $P' \to \P(1,n)$. (The proof of part
\eqref{part:mainthm2} of the theorem is independent of part
\eqref{part:mainthm1}). The bundles $P$ and $P'$ are isomorphic by the
following argument. The limit $u(\infty,\theta):=\lim_{r \to
  \infty}u(r,\theta)$ is well-defined and satisfies
$u(\infty,\theta)=k(\theta)u(\infty,0)$ for a path $k:[0,2\pi]\to K$,
where $k(0)=\Id$ and $k(2\pi)^n=\Id$. By Remark
\ref{rem:orbibundleinv}, the bundle $P$ is uniquely determined by the
path $k$. The bundle $P'$ is analogously determined by a path given by
$e^{i\xi_\infty}u$, this path is same as $k$. Since
$e^{i\xi_\infty}(A,u)$ is a $p$-bounded gauged holomorphic map on $P$,
$e^{i\xi_\infty}u$ is in $W^{1,p}$ on a neighborhood of $\infty$ in
$\tU_1$, and so the same is true of $\xi_\infty$, therefore
$e^{i\xi_\infty} \in \G(P)_{\on{bd}}$.

To finish the proof Theorem \ref{thm:main} \eqref{part:mainthm1}, it remains to show that the vortex
$e^{i\xi_\infty}(A,u)$ is unique up to gauge transformations. This is
proved separately in Proposition \ref{prop:uniqueaffinevort} using
Theorem \ref{thm:main} \eqref{part:mainthm2}.
\end{proof}

\subsection{Producing a limit modulo bubbling}

The main result of this section is Proposition \ref{prop:seq1}, which
can be described informally as follows.  We choose a sequence of balls, or annuli
in the orbifold case, that exhaust the complex line. Given a stable
gauged holomorphic map on the weighted projective space, there exists
a sequence of gauge transformations that transform the gauged
holomorphic map to a sequence of vortices on these increasing balls
(resp. annuli). This sequence of vortices converges modulo bubbling to
a limit vortex defined on all of the complex line.

\begin{proposition} \label{prop:seq1} {\rm (A limit modulo bubbling)}
  Suppose $(A,u)$ is a stable gauged holomorphic map defined on a
  principal $K$-bunde $P \to \P(1,n)$. Let $p>2$ and $R_i \to \infty$
  be an increasing sequence, with $R_1 \geq 2$. We denote by
  $A(1/R_i,R_i) \subset \C$ an annulus of radii $1/R_i$ and $R_i$.
  \begin{enumerate}
  \item \label{part:seq1_orb1} There exists a sequence of complex
    gauge transformations $g_i \in H^2(B_{R_i}, G)$ such that the
    restriction of $(A_i,u_i):=g_i(A,u)$ to the annulus $A(1/R_i,R_i)$ is a
    vortex whose energy is uniformly bounded. The curvature
    $\Mod{F_{g_iA}}_{L^2(B_{1/R_i})}$ is also uniformly bounded.
 
  \item \label{part:seq1_orb2} Let $Z$ be the set of points $z \in \C
    \bs \{0\}$ for which there is a sequence $z_i \to z$ such that
    $|\d_{A_i}u_i(z_i)| \to \infty$ as $i \to \infty$. Then, the set
    $Z$ is finite and there exists a sequence of gauge
    transformations $k_i\in H^2(B_{R_i},K)$ and a vortex
    $(A_\infty,u_\infty) \in (H^1_\loc \times H^2_\loc)(\C)$ such that
    \begin{enumerate}
    \item[(1)] $k_iA_i$ converges to $A_\infty$ weakly in $H^1$, and
      strongly in $L^p$ on compact subsets of $\C$.
    \item[(2)] $k_iu_i$ converges to $u_\infty$ weakly in $H^2$ and
      strongly in $W^{1,p}$ on compact subsets of $\C\bs (Z\cup
      \{0\})$.
    \end{enumerate}
  \end{enumerate}

  The vortex $(A_\infty,u_\infty)$ has finite energy and bounded
  image.
\end{proposition}

We outline the proof of Proposition \ref{prop:seq1}, starting with the manifold
case. In this case, the stable gauged holomorphic map is defined on
the projective line. We use a sequence of metrics that interpolate
between the Euclidean metric on the affine line and the Fubini-Study
metric. In particular, the sequence of metrics is constructed so that
the sequence of increasing balls in the Proposition respectively have
the Euclidean metric on them. With respect to each metric in the
sequence, one can apply the heat flow result Theorem \ref{thm:hk} to the
given gauged holomorphic map. For most (all but finite) metrics in the
sequence, the limit of the heat flow is a vortex under the respective
metric. The sequence of vortices, when respectively restricted to the
sequence of increasing balls, satisfies the necessary energy bound,
and hence we obtain the convergence in the sense of part
\eqref{part:seq1_orb2} of Proposition \ref{prop:seq1}.

We remark that in the case that the git quotient $X \qu G$ is a
manifold, a slightly weaker result suffices for the proof of Theorem
\ref{thm:main}: The elements of the sequence $(A_i,u_i)$ can be taken
to be vortices on the balls $B_{R_i}$, instead of annuli.  The set $Z$
of bubbling points would then be a subset of $\C$, not of $\C \bs
\{0\}$.  The additional curvature bound
$\Mod{F_{g_iA}}_{L^2(B_{1/R_i}}$ would not be required. In part
\eqref{part:seq1_orb2}, the convergence of the sequence $k_iu_i$ would
be in compact subsets of $\C \bs Z$.

In the orbifold case, we start with a stable gauged holomorphic map
defined on the weighted projective line. The proof differs in this
case because the heat flow does not apply when the base space is an
orbifold. To get around this issue we work on a cover of the weighted
projective line that is ramified at the origin. The cover is
bi-holomorphic to the projective line. Unfortunately, the lift of the
Euclidean metric to this cover degenerates at the ramification
point. We construct a sequence of metrics by interpolating between the Fubini-Study metric near
infinity, a lift of the Euclidean metric on a sequence of increasing
annuli and the Euclidean metric scaled by a factor near the
origin. Repeating the heat flow process produces a sequence of
vortices on the increasing annuli. The sequence converges in the sense
of the Proposition on the complex line punctured at the
origin. Additional technical details are required to remove the
singularity at the origin for the limit vortex.

\begin{definition}
  {\rm (Ramified cover of the orbifold line)}
  \label{def:ramcover} We denote by $\tilde \P$ an $n$-cover of
  $\P(1,n)$ that is ramified at $0$. The cover $\tilde \P$ is
  bi-holomorphic to the projective line $\P^1$ and the covering map is
$$\pi:\tilde\P \to \P(1,n) \quad [x:y] \mapsto [x^n:y].$$ 
\end{definition}

To see that this map is well-defined, recall that $\P(1,n)$ can be
defined as
$$\P(1,n):=(\C^2 \bs \{0\})/\sim, \quad (x,y) \sim (\lambda^n
x,\lambda y) \,\,\forall \lambda \in \C^\times.$$ It is useful to have
chart-based description of $\tilde \P$. Recall $\P(1,n)$ is
constructed using two charts $\tU_1$ and $U_2$ (see
\eqref{eq:charts}). The cover $\tilde \P$ is made up of charts
$\tU_1$, $\tU_2$.  The projection $\pi:\tU_2 \to U_2:=w \mapsto w^n$
is an $n$-cover ramified at $0$. We consider $\pi^*(A,u)$ and perform
all the following steps maintaining symmetry under the $\Z_n$-action.

\begin{figure}
  \centering \scalebox{.5}{ \input{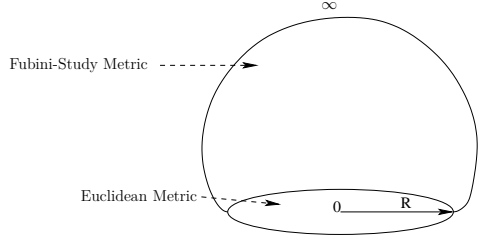}}
  \caption{The metric $\dvol_R$ on $\P^1$}
  \label{fig:dvolR}
\end{figure}

Next, we describe a family of metrics on the cover of the weighted
projective line that agrees with a lift of the Euclidean metric on a
sequence of exhausting subsets.  The Euclidean area form on $\C
\subset \P(1,n)$ pulls back to an area form
$$ \dvol_{\orb} = n^2(|x|^2+|y|^2)^{n-1} dx \wedge dy$$
on $\C \subset \tilde \P$. In case $n>1$, the corresponding metric degenerates at
the origin.  We define a metric $\dvol_R$ that is dependent on a parameter $R>0$
(see \eqref{eq:dvolR}) and is equal to
$\dvol_{\orb}$ on the annulus $A(1/R,R)$ of radii $1/R$ and $R$, it is
equal to the Fubini-Study metric in a neighborhood of $\infty$ and a
Euclidean metric in a neighborhood of $0$. The modification near the
origin is to remove the degeneracy of the metric.

\begin{definition} {\rm (A family of metrics on the orbifold line)}
  Let $\eta:\C \to [0,1]$ be a radially symmetric cut-off function
  that is $1$ in the unit ball $B_1$ and $0$ on $\C \bs B_2$. For any
  $R>0$, define a cut-off function
  \begin{equation}\label{eq:cutoff}
    \eta_R:\C \to [0,1]\quad \eta_R(x):=\eta(x/R).
  \end{equation}
  Let $\dvol_{\on{Euc}}$ denote the Euclidean metric on $\C \subset \tilde
  P$ and $\dvol_{\!\mathsmaller{\on{FS}}}$ be the Fubini-Study metric on
  $\tilde \P$. For any $R>2$, define a metric on $\dvol_R$ on $\tilde
  \P$ as
  \begin{equation}\label{eq:dvolR}
    \begin{split}  
      \dvol_{\!\mathsmaller R}&:= (1-\eta_{\!\mathsmaller R})\dvol_{\!\mathsmaller{\on{FS}}} + \eta_{\!\mathsmaller R}(1-\eta_{\!\mathsmaller{1/2R}})\dvol_{\orb} + \eta_{\!\mathsmaller{1/2R}} n^2R^{-2n+2}\dvol_{\on{Euc}}\\
      \nonumber &=\Bigl( \frac {1-\eta_{\!\mathsmaller R}}
      {(1+x^2+y^2)^2} +\eta_{\!\mathsmaller
        R}(1-\eta_{\!\mathsmaller{1/2R}}) n^2(x^2+y^2)^{n-1} +
      \eta_{\!\mathsmaller{1/2R}}n^2R^{-2n+2} \Bigr) dx \wedge dy,
    \end{split}
  \end{equation}
  where $(x,y)$ are Euclidean coordinates on $\C \subset \tilde \P$.
\end{definition}
We remark that in case the git quotient is a manifold, $n=1$, the family of metrics
simplifies to
\begin{equation}
  \label{simp}
  \dvol_{\!\mathsmaller R}:=(1-\eta_{\!\mathsmaller R})\dvol_{\!\mathsmaller{\on{FS}}} + \eta_{\!\mathsmaller R}\dvol_{\on{Euc}}= \left( \frac {1-\eta_{\!\mathsmaller R}} {(1+x^2+y^2)^2} +
    \eta_{\!\mathsmaller R}\right) dx\wedge dy.\end{equation}

The family of metrics described above satisfies
the following uniform bound:

\begin{lemma}
  \label{lem:dvolRbd} Let $\Om$, $B \subset \tilde \P$ be small open
  neighborhoods of $0$ and $\infty$ respectively. Suppose
  $f_{\!\mathsmaller R}:\tilde \P \to \R_{\geq 0}$ be given by the
  relation $\dvol_{\!\mathsmaller R}=f_{\!\mathsmaller
    R}\dvol_{\!\mathsmaller{\on{FS}}}$. There exists a constant $c$ such
  that
  \begin{align*}
    |f_{\!\mathsmaller R}| &\leq c \text{ on } \Om,& c^{-1} &\leq
    |f_{\!\mathsmaller R}| \text{ on } B,& c^{-1} \leq
    |f_{\!\mathsmaller R}|&\leq c \text{ on } \tilde \P \bs (\Om \cup
    B).
  \end{align*}
\end{lemma}
\begin{proof}
  To prove the first relation, suppose $\Om$ is contained in a ball of
  size $r$ about $0$. Consider values of $R$ for which $R>r>\frac 1
  R$. On $\Om$, $\eta_{\!\mathsmaller R}=1$. So, we can write
  \begin{align*}
    \dvol_{\!\mathsmaller R}&:=  (1-\eta_{\!\mathsmaller{1/2R}})\dvol_{\orb} + \eta_{\!\mathsmaller{1/2R}} n^2R^{-2n+2}\dvol_{\on{Euc}}\\
    & =\Bigl((1-\eta_{\!\mathsmaller{1/2R}}) n^2(x^2+y^2)^{n-1} +
    \eta_{\!\mathsmaller{1/2R}}n^2R^{-2n+2} \Bigr) dx \wedge dy.
  \end{align*}
  We introduce another function $g_{\!\mathsmaller
    R}:=f_{\!\mathsmaller R}(1+x^2+y^2)^{-2}$ on $\Om$, so that
  $\dvol_{\!\mathsmaller R}=g_{\!\mathsmaller R} dx \wedge dy$. Then,
  $g_{\!\mathsmaller R} \leq n^2r^{2n-2}$ on $\Om$, therefore there is
  an upper bound on $f_{\!\mathsmaller R}$ also.

  Next, consider the second relation. Suppose $r$ is such that $B$ is
  contained in $\tilde \P \bs B_r$. Consider values of $R$ such that $\frac
  1 R < r$. On $B$, the term $\eta_{\!\mathsmaller 1/2R}$ is $0$, then,
  \begin{align*}
    \dvol_{\!\mathsmaller R}&= (1-\eta_{\!\mathsmaller R})\dvol_{\!\mathsmaller{\on{FS}}} + \eta_{\!\mathsmaller R}\dvol_{\orb}  \\
    &=\Bigl( \frac {1-\eta_{\!\mathsmaller R}} {(1+x^2+y^2)^2}
    +\eta_{\!\mathsmaller R} n^2(x^2+y^2)^{n-1}\Bigr) dx \wedge dy,
  \end{align*}
  Since $\dvol_{\!\mathsmaller{\on{FS}}}$ is smaller than $\dvol_{\orb}$ on
  $B$, we have $f_{\!\mathsmaller R} \geq 1$.

  The last relation follows from the fact that for large enough $R$,
  the metric $\dvol_R$ is equal to $\dvol_{\orb}$ on $\tilde \P \bs
  (\Om \cup B)$.
\end{proof}

The following is a
preparatory
result required to prove the main result of this section (Proposition
\ref{prop:seq1}). It says that given a stable gauged holomorphic map
on a weighted projective line, it can be complex gauge transformed in
a way that its energy is uniformly bounded with respect to the family
of metrics described above.

\begin{proposition} \label{prop:prepare} Let $(A,u)$ be a gauged
  holomorphic map defined on a principal bundle $P \to \P(1,n)$, and
  suppose $u(\infty) \in X^{\ss}$. Then there exists a smooth complex
  gauge transformation $g \in \G(P)$ such that
  \begin{equation}\label{eq:preparebd}
    \sup_R E_R(\pi^*(g.(A,u)))<\infty. 
  \end{equation}
  Here $\pi:\tilde \P \to \P(1,n)$ is the projection and $E_R$ is the
  energy of the gauged holomorphic map with respect to the metric
  $\dvol_R$ on $\tilde \P$.
\end{proposition}
\begin{proof}
  We first find a smooth $g \in \G(P)$ such that
  \begin{align}\label{eq:preparecond}
      F_{gA} &\equiv 0  \text{ in a neighborhood of $0$,}&
      \Phi(gu) &\equiv 0 \text{ in a neighborhood of $\infty$.}
  \end{align}
  For the first condition, we use Theorem \ref{thm:YMbdry} on a
  neighborhood $\Om$ of $0$ in $\P(1,n)$. This produces a unique
  element $s:\Om \to P(\k)$ satisfying $s|_{\partial \Om}\equiv 0$
  such that $e^{is}A$ is a flat connection on $\Om$. The element $s$
  is smooth by Lemma \ref{lem:complexgreg}.  By taking $\Om$ to be a
  neighborhood of $0$, and using a cut-off function, we can produce
  $g_1 \in \G(P)$ that makes the connection flat in a neighborhood of
  $0$ and $g_1$ is identity away from this neighborhood.  For the
  second condition $\Phi(gu)=0$, observe that there is a neighborhood
  $B$ of $\infty$ in $\P(1,n)$ that is mapped by $u$ to $X^\ss$. For
  any $x \in X^\ss$, there is a unique $s \in \k$ such that
  $\Phi(e^{is}x)=0$, and $s$ varies smoothly with $x$. Let $\tilde
  B:=\pi^{-1}(B) \subset \tU_1$ be the lift of $B$. We work with the
  trivialization of $P$ over $\tU_1$ as in
  \eqref{eq:znsymmetry}. There exists $\xi:\tilde B \to \k$ such that
  $\Phi(e^{i\xi}u)\equiv 0$ on $\tilde B$. Recall $u$ satisfies
  $\Z_n$-equivariance $u \circ \sig_n=\mu u$. The uniqueness of $\xi$
  implies
  \begin{equation} \label{eq:xisym} \xi \circ \sig_n = \Ad_\mu \xi.
  \end{equation}
  By using a cut-off function, $\xi$ can be extended to all of $\tU_1$
  so that it vanishes away from $\tilde B$ and still satisfies the
  symmetry relation \eqref{eq:xisym}.  The element $g:=g_1e^{i\xi}$
  satisfies the conditions in \eqref{eq:preparecond}, by construction.

  Using Lemma \ref{lem:dvolRbd}, we compare the $R$-dependent metric
  with the Fubini-Study metric.  Conformally rescaling the metric by a
  scalar-valued function greater than one has the effect of decreasing
  the $L^2$ norm of the two-form $F_A$ and increasing the $L^2$ norm
  of the zero-form $\Phi(u)$. The $L^2$ norm of the one-form $\d_Au$ is
  independent of the metric.  By \eqref{eq:preparecond} we have that
  for any $R$,
  $$ E_R(\pi^*(g.(A,u)))< c E_{\on{FS}}(\pi^* (g.(A,u)) $$
  where $E_{\on{FS}}$ denotes the energy defined using the Fubini-Study
  metric.  Hence the complex gauge transformation $g$ satisfies the
  bound \eqref{eq:preparebd}.
\end{proof}
\begin{proof}{\rm (Proof of Proposition \ref{prop:seq1}
  when the git quotient is a manifold)}
  The proof of the first part
  of the proposition is by applying the heat flow Theorem \ref{thm:hk}
  on the given stable gauged holomorphic map with respect to a
  sequence of metrics described in \eqref{simp}.  Given a gauged
  holomorphic map $(A,u)$ on a principal bundle $P \to \P^1$ that
  satisfies $u(\infty) \in X^\ss$, by Proposition \ref{prop:prepare}
  we can assume
  \begin{equation}\label{eq:energyRbd}
    E_R(A,u)<e_0 \quad \forall R>0,
  \end{equation}
  where $E_R$ is the energy with respect to the metric $\dvol_R$. As
  $R$ increases, the volume $\vol_R(\P^1)$ increases to infinity. So,
  there exists $r_0$ such that $E_R(A,u) \leq e_0 \leq
  c_0^2\vol_R(\P^1)$ for all $R\geq r_0$, where $c_0$ is as defined in
  \eqref{eq:c0def}. By dropping some initial terms of the given
  sequence $\{R_i\}$, we have $R_i \geq r_0$ for all $i$. Theorem
  \ref{thm:hk} is applicable for the gauged holomorphic map $(A,u)$ on
  $\P^1$ with respect to the metric $\dvol_{R_i}$. By Theorem
  \ref{thm:hk}, $(A,u)$ is complex gauge equivalent to a vortex
  $(A_i,u_i)$ with respect to the metric $\dvol_{R_i}$ via a complex
  gauge transformation $g_i \in \G(P)_{W^{2,p}}$. On the ball
  $B_{R_i}$, $(A_i,u_i)$ is a vortex with respect to the Euclidean
  metric.  By modifying each $g_i$ by a gauge transformation on
  $B_{R_i}$, we may assume that $(A_i,u_i)$ is smooth on $B_{R_i}$ (by
   \cite[Theorem 3.1]{CGMS}). This finishes the proof of part
  \eqref{part:seq1_orb1} of the Proposition \ref{prop:seq1}.

  We now have a sequence of gauged holomorphic maps on the projective
  line whose restriction to a sequence of exhausting subsets of the
  complex line is a vortex with respect to the Euclidean metric. We
  show that the restrictions of the sequence of maps on these subsets
  satisfy the hypothesis for the Gromov convergence result (Lemma
  \ref{lem:vortconv}).  The gauged maps $(A_i,u_i)$ are all homotopic
  to each other, so they have the same equivariant homology class $[u]
  \in H_2^K(X)$.  Then, by \cite[Lemma 2.7]{CGMS}, there is a
  compact set $S \subset X$ that contains the images of all the
  $u_i$. The pair $(A_i,u_i)$, when restricted to $B_{R_i}$, is a
  vortex with respect to the Euclidean metric. It satisfies
  $E(A_i,u_i,B_{R_i}) \leq E_{R_i}(A_i,u_i,\P^1) \leq
  E_{R_i}(A,u,\P^1) \leq e_0$. The first inequality follows from the
  definition of the metric $\dvol_{R_i}$, the second one from the fact
  that heat flow decreases energy (see Lemma \ref{lem:energytop}) and
  the last one follows from \eqref{eq:energyRbd}.

  The proof of the Proposition can be completed by applying Gromov
  convergence for vortices (Lemma \ref{lem:vortconv}).  The sequence
  $\{(A_i,u_i)\}$ is a sequence of vortices on $B_{R_i}$, which
  exhaust $\C$. The sequence satisfies an energy bound, so Lemma
  \ref{lem:vortconv} is applicable.  So, there is a subsequence of
  $(A_i,u_i)$ (still denoted by the same subscripts), a sequence of
  gauge transformations $k_i \in H^2(B_{R_i})$, a finite set $Z
  \subseteq \C$ and a finite energy vortex $(A_\infty,u_\infty)$ on
  $\C$ so that $k_iA_i \weakto A_\infty$ in $H^1$ on compact subsets
  of $\C$ and $k_iu_i \weakto u_\infty$ in $H^2$ on compact subsets of
  $\C \bs Z$. 
\end{proof}

\begin{proof}{\rm (Proof of Proposition \ref{prop:seq1} when the git quotient is an orbifold.)}
  We work with a lift of the given stable gauged holomorphic map to
  the ramified cover of the weighted projective line. We first state
  some symmetry properties of this lift. Suppose $(A,u)$ is the given
  gauged holomorphic map $(A,u)$ on a principal bundle $P \to
  \P(1,n)$. The lift of this map to the bundle $\pi^*P \to \tilde \P$
  is denoted $(\tilde A, \tilde u)$. A trivialization of the bundle $P
  \to \P(1,n)$ on the charts $\tU_1$, $U_2$ (see \eqref{eq:equivbund})
  lifts to a trivialization of $\pi^*P \to \tilde \P$ on $\tU_1$ and
  $\tU_2$. We fix such a trivialization.  On the chart $\tU_2 \simeq
  \C \subset \tilde \P$, the lift $(\tilde A, \tilde u)$ is symmetric
  under the $\Z_n$-action, or in other words $(\tilde A, \tilde u)$ is
  preserved by $\frac {2\pi} n$-rotations of the domain. On the
  trivialization over $\tU_1 \simeq \tilde \P \bs \{0\}$, we have
  $\sig_n^*(\tilde A, \tilde u) = \mu (\tilde A, \tilde u)$.

  We now produce a sequence of vortices on a sequence of exhausting
  annuli. We also prove energy bounds in this sequence so that Gromov
  convergence is applicable. We define another increasing sequence
  $\tilde R_i:=R_i^{1/n}$, and remark that the annulus $A(1/R_i,R_i)$
  in $\P(1,n)$ lifts to the annulus $A(1/\tilde R_i,\tilde R_i)$ in
  $\tilde \P$.  By Proposition \ref{prop:prepare}, we can assume
  $$E_{\tilde R}(\tilde A,\tilde u, \tilde \P)<e_0 \quad \forall \tilde R>0,$$ 
  where $E_{\tilde R}$ denotes the energy with respect to the
  $\dvol_{\tilde R}$ metric. As in the case when the git quotient is a
  manifold, after dropping some terms in the sequence $\{\tilde
  R_i\}$, we may assume that $E_{\tilde R_i}(\tilde A,\tilde u) \leq
  e_0 \leq c_0^2\vol_{\tilde R_i}(\tilde \P)$, which ensures that
  Theorem \ref{thm:hk} is applicable. Using Theorem \ref{thm:hk}, we
  get a sequence of complex gauge transformations $e^{i\tilde
    \xi_i}\in W^{2,p}(\G(\pi^*P))$ such that $e^{i\tilde \xi_i}(\tilde
  A,\tilde u)$ are $\dvol_{\tilde R_i}$-vortices on $\tilde \P$. The
  starting energy $E_{\tilde R_i}(A,u)$ is bounded and heat flow
  decreases energy (see Lemma \ref{lem:energytop}). Therefore, the
  elements in the sequence $e^{i\tilde \xi_i}(\tilde A,\tilde u)$ have
  bounded energy with respect to the metric $\dvol_{\tilde R_i}$.
  Further, the gauged maps $e^{i\tilde \xi_i}(\tilde A, \tilde u)$ are
  homotopic to each other, so by \cite[Lemma 2.7]{CGMS}, there is a
  compact set $S \subset X$ that contains the images of the maps
  $e^{i\tilde \xi_i}\tilde u$.  By Theorem \ref{thm:hk}, the element
  $\tilde \xi_i$ is unique, so the complex gauge transformation
  $e^{i\tilde \xi_i}$ is symmetric under the $\Z_n$ action, and
  descends to the weighted projective line, we denote
  $\xi_i:=\pi_*\tilde \xi_i$. The same is therefore true of the pair
  $e^{i\tilde \xi_i}(\tilde A,\tilde u)$, and it descends to the pair
  $e^{i\xi_i}(A,u)$ defined on $\P(1,n)$. The metric $\dvol_{\tilde
    R_i}$ descends to the Euclidean metric on the annulus
  $A(1/R_i,R_i) \subset \P(1,n)$, so $e^{i\xi_i}(A,u)$ is a vortex
  with respect to the Euclidean metric on this annulus with an energy
  bound:
  \begin{equation}\label{eq:eucenergy1}
    E(e^{i\xi_i}(A,u); A(1/R_i,R_i))=\frac 1 n E_{\dvol_{\tilde R_i}}(e^{i\tilde \xi_i}(\tilde A,\tilde u);A(1/\tilde R_i, \tilde R_i)) \leq  e_0/n.
  \end{equation}
  However, in the neighborhood of the origin, we do not have control
  over the $L^2$ norm of the curvature with respect to the Euclidean
  metric. We apply Proposition \ref{prop:almostvortexnear0} to the
  sequence $e^{i\tilde \xi_i}(\tilde A,\tilde u)$ on $\tilde \P$. This
  gives a sequence of $H^2$ complex gauge transformations $e^{i\tilde s_i}$
  on $\tilde \P$ supported in the balls $B_{1/\tilde R_i}$ that are
  $\Z_n$-symmetric and satisfy the following: if $s_i:=\pi_*\tilde
  s_i:\C \to \k$, then the curvature norms
  $\Mod{F_{e^{is_i}e^{i\xi_i}A}}_{L^2}$ are uniformly bounded. Part
  \eqref{part:seq1_orb1} of Proposition \ref{prop:seq1}
   is proved by the complex gauge transformations
  $g_i:=e^{is_i}e^{i\xi_i}:B_{R_i} \to G$. Set $(A_i,u_i):=g_i(A,u)$
  on $B_{R_i}$.  Since $s_i$ is supported in $B_{1/R_i}$, on the
  annulus $A(1/R_i,R_i)$, $(A_i,u_i)=g_i(A,u)$ is still a
  vortex that satisfies
 \begin{equation}
   \label{eq:eucenergy2}
   E(A_i,u_i;A(1/R_i,R_i)) \leq c, \quad \Mod{F_{A_i}}_{L^2(B_{1/R_i})} \leq c.
 \end{equation}

 We apply Gromov convergence on the sequence of vortices constructed
 above that are defined on a sequence of exhausting annuli. Combining the equations in \eqref{eq:eucenergy2}, we get a uniform bound on the energy
 of the connections $\Mod{F(A_i)}_{L^2(B_{R_i})}^2$. By Uhlenbeck compactness for non-compact domains (
 \cite[Theorem A']{Weh:Uh}), after passing to a subsequence, there is a sequence
 of gauge transformations $k_i \in H^2(B_{R_i})$ and a connection
 $A_\infty$ on the trivial bundle $\C \times K$ such that $k_iA_i
 \weakto A_\infty$ in $H^1$ in compact subsets of $\C$. We apply Lemma
 \ref{lem:vortconv} to the sequence of vortices on domains $B_{R_i}\bs
 B_{1/R_i}$.  The Lemma implies that there is a finite set $Z \subset
 \C \bs \{0\}$ and a finite energy limit vortex $(A_\infty,u_\infty)$,
 where $A_\infty$ is the limit connection found earlier and $u_\infty$
 is defined on $\C \bs \{0\}$ such that $k_iu_i \weakto u_\infty$ in
 $H^2$ on compact subsets of $\C \bs (Z \cup \{0\})$. 
    
 To finish the proof of the Proposition, it remains to remove the
 singularity of the limit map at the origin. Choose a small positive
 number $r>0$. The pair $(A_\infty,u_\infty)$ is a vortex on $B_r \bs
 \{0\}$. Since the image $k_iu_i(B_r \bs B_{R_i})$ is contained in a compact set $S
 \subset X$ for all $i$, the image of $u_\infty$ is also contained in
 $S$. By Lemma \ref{lem:toflat}, there is a complex gauge
 transformation $g \in H^2(B_r)$ such that $gA_\infty$ is the trivial
 connection. Complex gauge transformations preserve holomorphicity, so
 $\delbar(gu_\infty)=0$ on $B_r \bs \{0\}$ and the image
 $gu_\infty(B_R \bs \{0\})$ is contained in a compact set. Further
 $gu_\infty$ is smooth. By the removable singularity theorem of
 complex analysis, the map $gu_\infty$ extends smoothly over $0$, and
 hence $u_\infty$ extends over $0$ as a $H^2$ map.
\end{proof}

\subsection{Finding complex gauge transformations to relate nearby vortices}

Given a sequence of converging vortices on a ball in the complex
plane that are in the same complex gauge equivalence class, we
address the question of whether the limit modulo bubbling of the sequence is also
in the same complex gauge equivalence class. As part of the
hypothesis, we are given that the limit is in the same complex
gauge equivalence class in an annulus. In particular the limit is
determined by a limit complex gauge transformation. In the following
proposition we extend the limit complex gauge transformation to the
entire ball. If one naively tries to take the limit of complex gauge
transformations generating the sequence of vortices, one is obstructed
by (a) the bubbling points (represented by $Z$ below) and (b) points in a semistable orbit coverging to an unstable point in the closure of the orbit.

\begin{proposition} \label{prop:extend_inball} Let $0<r<R$ and 
$B_r \subset B_R \subset \C$ be open balls of radii $r$ and $R$.  Let $Z \subset B_r$ be a finite
  set and $(A_i,u_i)$ be a sequence of vortices in class $(L^p \times
  W^{1,p})(B_R)$ that satisfy the following.
  \begin{enumerate}
  \item The sequence $(A_i,u_i)$ converges to a vortex
    $(A_\infty,u_\infty) \in (L^p \times W^{1,p})(B_R)$ in the
    following sense:
    \begin{align}\label{eq:convsense}
      A_i &\xrightarrow{L^p(B_R)} A_\infty,& u_i
      &\xrightarrow{W^{1,p}(Q)} u_\infty \text{ for all compact
        subsets $Q \subset B_R \bs Z$}.
    \end{align}
  \item There are complex gauge transformations $g_i \in
    W^{1,p}(B_R,G)$ such that $g_i(A_0,u_0)=(A_i,u_i)$.
  \item There is also a complex gauge transformation $g_\infty:\ol B_R
    \bs B_r \to G$ such that $g_i$ converges to $g_\infty$
    in $W^{1,p}(B_R \bs B_r)$ and $(A_\infty,
    u_\infty)=g_\infty(A_0,u_0)$ on $\ol B_R \bs B_r$.
  \item For all $i$, including $i=\infty$, $u_i(\ol B_R \bs
    B_r) \subset X^\ss $.
  \end{enumerate}
  Then, the complex gauge transformation $g_\infty$ can be defined on
  all of $\ol B_R$ in a way that $(A_\infty,u_\infty)=g_\infty(A_0,u_0)$ and the sequence $g_i$ converges to $g_\infty$ weakly in $W^{1,p}(B_R)$.
\end{proposition}
\begin{proof}
  In Step 1 and 2 of the proof, we adjust the sequence of vortices,
  via a sequence of small complex gauge transformations, to produce a
  new sequence of vortices that are equal to each other on the
  boundary upto unitary gauge equivalence. Then, in Step 3, we show that
  the vortices in the new sequence, which still has the same limit, are related to each other by unitary gauge
  transformations. Finally, in Step 4, by a gauge theoretical argument, we find a
  limit of the unitary gauge transformations, that indeed corresponds
  to the limit vortex. Let $N(\partial B_R)$ denote a closed neighborhood of $\partial B_R$ disjoint from $\ol B_r$.

  {\sc Step 1}: {\em There exists a sequence of complex gauge
    transformations $g_i' \in W^{1,p}(B_R,G)$, each of whose terms is
    respectively equal to $g_\infty g_i^{-1}$ on $N(\partial B_R)$
    and is equal to $\Id$ on $B_r$.  Further $g_i' \to \Id$ in
    $W^{1,p}(B_R,G)$. Consequently, the gauged holomorphic maps
    $(A_i',u_i'):=g_i(A_i,u_i)$ are equal to $(A_\infty,u_\infty)$ on
    $N(\partial B_R)$ and $(A_i',u_i')$ is equal to $(A_i,u_i)$ on $B_r$.}\\
  By the hypothesis of the proposition, $g_\infty g_i^{-1}$ converges
  to $\Id$ in $W^{1,p}(B_R \bs B_r)$. So, for large $i$, the image
  of $g_\infty g_i^{-1}$ is contained in a neighborhood of $\Id$ where the exponential map
  $e:\g \to G$ is injective. This means we can write $e^{\zeta_i}=g_\infty
  g_i^{-1}$, where $\zeta_i:\ol B_R \bs B_r \to \g$. Let
  $\eta:\ol B_R \to [0,1]$ be a cut-off function that is equal to $1$ in a
  neighborhood of $\partial B_R$ and is equal to $0$ on $B_r$. The
  required complex gauge transformations are $g_i':=e^{i\eta
    \zeta_i}$.

  {\sc Step 2}: {\em There exists a sequence $\xi_i \in W^{1,p}(B_R,\k)$ so that $\xi_i|_{\partial B_R} =0$ and the gauged holomorphic maps
    $(A_i'',u_i''):=e^{i\xi_i}(A_i',u_i')$ are vortices on
    $\ol B_R$. Further $\xi_i \to 0$ in $W^{1,p}(B_R,\k)$.}\\
  This result is obtained by applying Proposition \ref{prop:orbitvortex} on the
  sequence $(A_i',u_i')$ defined on $\ol B_R$. First we check the hypothesis. The
  convergence of $(A_i,u_i)$ in \eqref{eq:convsense} together with the
  condition $g_i' \to \Id$ in $W^{1,p}(B_R,G)$ guarantees that the
  following is true:
  \begin{align}\label{eq:convsenseprime}
    A_i' &\xrightarrow{L^p(B_R)} A_\infty,& u_i'
    &\xrightarrow{W^{1,p}(Q)} u_\infty \text{ for all compact subsets
      $Q \subset \ol B_R \bs Z$}.
  \end{align}
  We also have $F_{A_i',u_i'}$ (recall terminology from
  \eqref{eq:vortex}) converges to $0$ in $W^{-1,p}$. This is because $(A_i',u_i')$ is
  a vortex on $B_r$ and on $\ol B_R \bs B_r$, the convergence follows
  using \eqref{eq:convsenseprime} and the fact that $Z \subset
  B_r$. Step 2 is proved by the conclusions of Proposition
  \ref{prop:orbitvortex}.

  {\sc Step 3}: {\em The vortices in the sequence $(A_i'',u_i'')$
    are unitary gauge equivalent to each other.}\\
  The vortices $(A_i'',u_i'')$ are complex gauge equivalent to each
  other, in particular, on $\ol B_R$,
  $$(A_i'',u_i'')=h_i(A_0,u_0), \quad \text{ where }h_i:=e^{i\xi_i}g_i'g_i \in W^{1,p}(B_R,G).$$
  By the
  construction of $g_i'$ in Step 1, we know that $g_i'g_i$ is equal to
  $g_\infty$ in $N(\partial B_R)$. Together with the fact
  that $\xi_i|_{\partial B_R}=0$, we can say that
  $h_ih_j^{-1}(\partial B_R) \subset K$ for any $i$, $j$. But then,
  by Proposition \ref{prop:atmost1vortex}, the vortices are gauge
  equivalent, in particular $h_ih_j^{-1}(\ol B_R) \subset K$ for all $i$,
  $j$.

  {\sc Step 4}: {\em Finishing the proof.}\\
  First, we claim that $A_i'' \to A_\infty$ in $L^p(B_R)$. This
  follows from the convergence of $A_i$, $g_i'$ and $\xi_i$ in the
  spaces $L^p$, $W^{1,p}$ and $W^{1,p}$ respectively, and by using the
  continuous action of $\G(P)$ on $\A(P)$ (see Lemma
  \ref{lem:gcactona}). We can write $(A_i'',u_i'')=h_i'(A_0'',u_0'')$
  and we know from Step 3 that $h_i'=h_ih_0^{-1}$ and that it is a
  unitary gauge transformation. By a standard gauge theoretic argument
  (Lemma \ref{lem:hausquot}), after passing to a
  subsequence $h_i' \weakto h_\infty'$ in $W^{1,p}(B_R)$ (and
  strongly in $C^0$), $A_\infty = h_\infty' A_0''$. This implies that
  the sequence $u_i''$, which is same as $h_i'u_0''$, converges to
  $h_\infty 'u_0''$ weakly in $W^{1,p}(B_R)$.  Also, $u_i''$
  converges in $W^{1,p}$ to $u_\infty$ on compact subsets of $\ol B_R \bs
  Z$, so $u_\infty=h_\infty' u_0''$. Therefore, on $\ol B_R$, using Figure \ref{fig:cgauge},

  \begin{equation*}
    (A_\infty,u_\infty)=h_\infty'(A_0'',u_0'')=h_\infty'e^{i\xi_0}g_0'g_0(A_0,u_0).
  \end{equation*}
  To finish the proof, we show that the complex gauge transformation $h_\infty'e^{i\xi_0}g_0'g_0$ is the weak $W^{1,p}(B_R)$ limit of the sequence $g_i$. Working in the weak $W^{1,p}(B_R)$-topology, we can write
$$h_\infty'e^{i\xi_0}g_0'g_0=\lim_i h_i'e^{i\xi_0}g_0'g_0 =\lim_i (e^{i\xi_i}g_i'g_i)h_0^{-1}e^{i\xi_0}g_0'g_0 = \lim_i (e^{i\xi_i}g_i'g_i).$$
Since the terms $g_i'$ and $e^{i\xi_i}$ converge to $\Id$ in $W^{1,p}(B_R)$, we have proved that $h_\infty'e^{i\xi_0}g_0'g_0$ is the weak $W^{1,p}(B_R)$-limit of the sequence $g_i$. So, $g_\infty:=h_\infty'e^{i\xi_0}g_0'g_0$ is an extension of the limit complex gauge transformation to all of $B_R$ and $g_\infty(A_0,u_0)=(A_\infty,u_\infty)$.
\begin{figure}
  \centering { \input{diag.pstex_t}}
  \caption{Complex gauge transformations used in proof of Proposition \ref{prop:extend_inball}.}
  \label{fig:cgauge}
\end{figure}
\end{proof}

\subsection{Some results involving complex gauge transformations}

The following sequence of Lemmas \ref{lem:transfconv} -
\ref{lem:complexgreg} were used in the proof of Theorem \ref{thm:main}
\eqref{part:mainthm1}. Lemma \ref{lem:transfconv} says that a sequence of
converging complex gauge equivalent maps to the semistable locus, after passing to a
subsequence, is generated by a converging sequence of complex gauge
transformations.
\begin{lemma}{\rm (Converging maps are related by converging complex gauge transformations)}\label{lem:transfconv}
  Suppose $\Sig \subseteq \C$ be the closure of a bounded open subset with smooth boundary and let maps $u_i :\Sig
  \to X^{\ss}$ converge in $W^{m,p}$ to $u_\infty:\Sig \to X^{\ss}$
  for some $m$, $p$ satisfying $mp>2$. Assume further that the maps
  are related by complex gauge transformations i.e. there exist $g_i \in W^{m,p}(\Sig,G)$ such that $u_i=g_iu_0$. Then, after passing to a subsequence $g_i$ converges to a limit $g_\infty$ in $W^{m,p}$ and $g_\infty u_0=u_\infty$.
\end{lemma}

\begin{proof}
  Locally, the Lemma is proved by using a slice of the group action on the target space.
  Let $x \in X \qu G$ and $y \in \pi_G^{-1}(x)$.
 We use a $G_y$-invariant slice
  $V \subset X^\ss$ containing $y$ such that there is a diffeomorphism
  $G \times_{G_y} V \to GV$. So, $G \times V \to GV$ is a
  $|G_y|$-cover.  Suppose $U \subset \Sig$ is such that $u_\infty(U)
  \subset V$, $u_i(U) \subset V$ for large $i$. Choose lifts
  $\tu_i=(\tilde g_i,v_i):U \to G \times V$ so that $\tilde g_i^{-1}\tilde
  g_j=g_i^{-1}g_j|_U$. Since $u_i \to u_\infty$, a subsequence of $\tu_i$
  converges in $W^{m,p}_{\on{loc}}$. (This is because: $G \times V \to GV$ is a
  local diffeomorphism and has smooth inverses locally.) So, $\tilde g_i$ converges and
  consequently $g_i$ converges to a limit $g_\infty:U \to G$. The set $\Sig$ can be covered
  by a countable number of sets of the form of $U$. Working with the chosen cover of
  $\Sig$, by successively passing to subsequences and using a diagonalization argument, we obtain a limit
  $g_\infty$ on all of $\Sig$. By the smooth action of $G$ on $X^\ss$, we get $u_\infty=g_\infty u_0$.
\end{proof}

The next Lemma shows that for two gauged holomorphic pairs whose
image lies in the semistable locus, if the maps are related by a
complex gauge transformation, then the same is true of the connections
also.
\begin{lemma}\label{lem:cgaugea} Let $\Sig \subset \C$ and $(A_i,u_i)$
  be vortices on $\Sig$ for $i=1,2$ such that $u_i(\Sig) \subset
  X^\ss$. Suppose there is a complex gauge transformation $g:\Sig \to
  G$ such that $gu_1=u_2$. Then, $gA_1=A_2$.
\end{lemma}
\begin{proof} The proof relies on the free infinitesimal action of the group on the target space of the maps.
  \begin{align}\label{eq:disjact}
    \k_x \cap J\k_x=\{0\}.
  \end{align}
  Let $a:=gA_1 - A_2$. Since both $(gA_1,u_2)$ and $(A_2,u_2)$ are
  holomorphic, we have $\delbar_{gA_1}u_2-\delbar_{A_2}u_2=a_{u_2}^{0,1}=0$. Write $a=a_xdx+a_ydy$,
  $a_x$, $a_y:\Sig \to \k$. Then, $a_x(u_2)+Ja_y(u_2)=0$ and by
  \eqref{eq:disjact}, $a=0$ on $\Sig$.
\end{proof}

\begin{lemma}
  \label{lem:complexgreg} {\rm (Regularity of complex gauge
    transformations)}
  Let $k \in \Z_{\geq 0}$ and $p>1$ be such that $kp>2$. Let $\Sig$ be a compact Riemann surface and $P \to \Sig$ be a
  principal $K$-bundle, and let $P_\C:=P \times_K G$ be the associated $G$-bundle. Suppose
  $A \in \A(P)$ is a smooth connection and $g \in \G(P)_{k,p}$ a
  complex transformation such that $gA$ is also smooth. Then $g$ is
  smooth.
\end{lemma}

\begin{proof} We use the relation on $P_\C$: $\delbar_{gA}=g \circ
  \delbar_A \circ g^{-1}$. The difference $a:=gA-A$ is smooth and
  \begin{align*}
    \delbar_{gA}-\delbar_A=a^{0,1}=g\delbar_A (g^{-1})=-(\delbar_A g)
    g^{-1}.
  \end{align*}
  Therefore $a^{0,1}g=-\delbar_A g$. The smoothness of $g$ follows by
  elliptic bootstrapping.
\end{proof}

\subsection{A technical result}

The result presented in this section is required only in the case of
an orbifold
git quotient and is used in the proof of Proposition \ref{prop:seq1}
above. In transforming a stable gauged holomorphic map to an affine
vortex, we recall that our first step was to produce a sequence of
vortices on a sequence of exhausting annuli. This was done using heat
flow on a ramified cover of the weighted projective line. The lift of
the Euclidean metric on the affine line degenerates at the
ramification point, which is chosen to be $0$. So the sequence of
metrics used to run the heat flow is altered near the origin, the new
sequence is called $\dvol_R$. The result of heat flow is a sequence of
$\dvol_R$-vortices. Now, when the norm of the curvature is taken with
respect to the lift of the Euclidean metric, it is no longer
bounded. Therefore, the $\dvol_R$-vortices have to be altered by a
sequence of complex gauge transformations supported near the
origin. This alteration, carried out in the Proposition below, makes
the $L^2$ norm of the curvature bounded with respect to the lift of
the Euclidean metric. This is useful in removing the singularity of
the limit vortex at the origin.

\begin{proposition} \label{prop:almostvortexnear0} Let $R_i \to
  \infty$ be an increasing sequence and $\tilde R_i:=R_i^{1/n}$. Let
  $(\tilde A_i, \tilde u_i)$ be a sequence of $\dvol_{\tilde
    R_i}$-vortices on the $K$-bundle $\pi^*P \to \tilde \P$ that are
  symmetric under the $\Z_n$-action, whose energies are bounded and
  there is a compact set that contains the images of the maps $\tu_i$.
  There exist complex gauge transformations $e^{i\tilde s_i} \in
  W^{2,p}_{\loc}(\G(\pi^*P))$ that are symmetric under the
  $\Z_n$-action, are equal to identity on $\tilde \P \bs B_{1/\tilde
    R_i}$ and for which
  $$\sup_i \Mod{F_{\pi_*(e^{i\tilde s_i}\tilde A_i)}}_{L^2(B_{R_i})}<\infty.$$
The complex gauge transformations satisfy $\Mod{\tilde s_i}_{C^0(B_{1/R_i},\k)}<c$.
\end{proposition}
\begin{proof} We explain the strategy of the proof. The Euclidean
  metric on the affine line in $\P(1,n)$ corresponds to the orbifold
  metric $\dvol_{\on{orb}}$ on the cover $\tilde \P$. For the given
  sequence of vortices, the curvature norm is bounded with respect
to the  $\dvol_{\tilde R_i}$ metric, which near the origin is bigger than the
  orbifold norm. Rescaling the metric by a factor smaller than one has
  the effect of increasing the $L^2$ norm of the curvature (a
  two-form). The Proposition is proved by a sequence of complex gauge
  transformations that make the connection flat in a small
  neighborhood (with radius of $1/2\tilde R_i$) of zero.

  We first produce the sequence of complex gauge transformations. We
  are working on balls in $\tilde \P$ about the origin with radius
  $1/\tilde R_i$. But, since we need uniform bounds on the sequence of
  complex gauge transformations, it is necessary to use the same
  domain to construct each element of the sequence. This is done by
  dilating the domains to a ball of constant radius.  For any $R>0$,
  let $\tilde B_R \subset \tilde U_2$ be the ball of radius $R$
  centered at $0$ in $\tilde \P$.  Let
  $$\sig_R: \tilde B_1 \to \tilde B_{1/R} \quad x \mapsto x/R$$
  be the dilation map. In the proof of this Proposition, $c$ will
  denote a constant that is independent of $i$, whose value may vary
  across appearances. In this proof, we use the symbol ``$\approx$''
  in the following sense: for two real non-negative $i$-dependent
  quantities $A$, $B$, we say $A \approx B$ if there exist constants
  $c_1$ and $c_2$ independent of $i$ such that $c_1A \leq B \leq
  c_2A$. We have
  \begin{equation}\label{eq:curvbd}
    \begin{split}
      \Mod{\sig_{\tilde R_i}^*F(\tilde A_i)}_{L^2(\tilde B_1)}&=\frac 1 {\tilde R_i} \Mod{F(\tilde A_i)}^{Euc}_{L^2(\tilde B_{1/\tilde R_i})} \approx \tilde R_i^{-2n+1} \Mod{F(\tilde A_i)}^{\dvol_{\tilde R_i}}_{L^2(\tilde B_{1/\tilde R_i})} \\
      & = \tilde R_i^{-2n+1} \Mod{\Phi(u_i)}^{\dvol_{\tilde
          R_i}}_{L^2(\tilde B_{1/\tilde R_i})} \leq c\tilde
      R_i^{-3n+1}.
    \end{split}
  \end{equation}
  For the second equality, we use the fact that in this region,
  $\dvol_{R_i}$ is the Euclidean metric scaled by a factor of
  $R_i^{-2n+2}$ times a positive function with values between $1$ and
  $2^{2n-2}$. The third equality follows from the fact that $(\tilde
  A_i,\tu_i)$ is a $\dvol_{\tilde R_i}$-vortex. For the last
  inequality, we observe that the images of $u_i$ are contained in a
  compact subset of $X$ and hence there is a $C^0$ bound on
  $\Phi(u_i)$, so the bound on the $L^2$ norm has a factor of the
  square-root of the volume of the domain. By Uhlenbeck's local
  theorem (\cite[Theorem 2.1]{Uh:compactness}), for large $i$, we can
  now choose a trivialization corresponding to each connection
  $\sig_{\tilde R_i}^*\tilde A_i$, under which the connection matrices
  are uniformly bounded and are invariant under the $\Z_n$-action,
  i.e.
  \begin{equation}\label{eq:connbd}
    \Mod{\sig_{\tilde R_i}^*\tilde A_i}_{H^1(\tilde B_1)} \leq c\tilde R_i^{-3n+1}.
  \end{equation}
  We remark that this requires a symmetric version of the local
  theorem that given a symmetric connection, there exists a symmetric
  gauge transformation that puts the connection in Coulomb gauge.  The
  symmetric version can be proved in a straightforward way by going
  through the steps of the proof of the local theorem.  By Lemma
  \ref{lem:toflat} and Remark \ref{rem:toflat_higher}, there is a
  constant $c$ such that for large enough $i$, there exists $\tilde
  \xi_i\in H^2(\tilde B_{1/\tilde R_i},\k)$ such that $\tilde
  \xi_i|_{\partial B_{1/\tilde R_i}}=0$, the connection $e^{i\tilde
    \xi_i}(\tilde A_i)$ is flat and
  \begin{equation}\label{eq:sw2pbd}
    \Mod{\sig_{\tilde R_i}^* \tilde \xi_i}_{H^2(\tilde B_1)} \leq c\Mod{\sig_{\tilde R_i}^* F(\tilde A_i)}_{L^2(\tilde B_1)} \leq c\tilde R_i^{-3n+1}.
  \end{equation}
  By the uniqueness of $\tilde \xi_i$, it is symmetric under the
  $\Z_n$-action.  The bound \eqref{eq:sw2pbd} implies a $C^0$ bound
  \begin{equation}\label{eq:sicobd}
    \Mod{\tilde
      \xi_i}_{C^0} \leq c\tilde R_i^{-3n+1},
  \end{equation} 
  and this bound is independent of metric. We cut-off $e^{i\tilde
    \xi_i}$ outside a ball of radius $1/\tilde R_i$ and define $\tilde
  s_i:=\eta_{1/2\tilde R_i}\tilde \xi_i$, where the cut-off function
  $\eta_{1/2\tilde R_i}$ is as defined in \eqref{eq:cutoff}. 
The required sequence of complex gauge transformations is
$e^{i\tilde s_i}$. The elements of the sequence inherit the $C^0$ bound on $\tilde \xi_i$ from
\eqref{eq:sicobd}.

We show that the curvature bound in the Proposition is satisfied with
the complex gauge transformations $e^{i\tilde s_i}$ chosen above.  We
continue to work with the dilated quantities $\sig_R^*\tilde A_i$,
$\sig_R^*\tilde s_i$ defined on the unit ball $\tilde B_1 \subset \tilde
\P$. The $H^2$ bound on $\sig_R^*\tilde \xi_i$ implies
\begin{equation}
  \label{eq:sh2bd}
  \Mod{\sig_{\tilde R_i}^* \tilde s_i}_{H^2(\tilde B_1)} \leq c\tilde R_i^{-3n+1}.
\end{equation}
For large $i$, the estimate in Lemma \ref{lem:gcactona} is applicable
for the action of the complex gauge transformations $\sig_{\tilde
  R_i}^* \tilde s_i$ on the connections $\sig_{\tilde R_i}^* \tilde
A_i$. Using \eqref{eq:sh2bd} and \eqref{eq:connbd} alongwith Lemma
\ref{lem:gcactona}, we can conclude
\begin{equation*}
  \Mod{\sig_{\tilde R_i}^* (e^{i\tilde s_i}\tilde A_i-\tilde A_i)}_{H^1(\tilde B_1)}<c\tilde R_i^{-3n+1}.
\end{equation*}
The curvature map $H^1 \ni A \mapsto F_A \in L^2$ is continuous. So,
\begin{equation*}
  \Mod{\sig_R^*(F_{e^{i\tilde s_i}\tilde A_i}-F_{\tilde A_i})}_{L^2(\tilde B_1)} \leq c\Mod{\sig_{\tilde R_i}^* (e^{\tilde s_i}\tilde A_i-\tilde A_i)}_{H^1(\tilde B_1)}<c\tilde R_i^{-3n+1}.
\end{equation*}
In the above equation, the contribution of quadratic terms involving
the connection can be ignored by adjusting the constant $c$, because
the term $e^{\tilde s_i}\tilde A_i-\tilde A_i$ has small norms for
large $i$ and so, the quadratic term is domainated by the linear term. Combining with \eqref{eq:curvbd}, we get
\begin{equation} \label{eq:transfcurvbd}
  \Mod{\sig_R^*F_{e^{i\tilde s_i}\tilde A_i}}_{L^2(\tilde B_1)}<c\tilde R_i^{-3n+1}.
\end{equation}

Finally we transport the bound on the curvature on $\tilde B_1$ to the
ball $B_{R_i} \subset \P(1,n)$. The quantities $\tilde s_i$ and
$\tilde A_i$ are symmetic under the $\Z_n$-action, so they descend via
$\pi_*$ to $\P(1,n)$.  Define
$$s_i :=\pi_*\tilde s_i, \ A_i:= \pi_* \tilde A_i.$$
We recall that $\pi$ maps the ball $\tilde B_{\tilde R_i}$ to $B_{R_i}$ in $\P(1,n)$. By construction of $s_i$, the connection $e^{is_i}A_i$ is flat in $B_{2^{-n}/R_i}$. So,
\begin{align*}
  \Mod{F_{e^{is_i}A_i}}_{L^2(B_{1/R_i})}^{\Euc}&=\Mod{F_{e^{is_i}A_i}}_{L^2(A(2^{-n}/R_i,1/R_i))}^{\Euc} = \Mod{F_{e^{i\tilde s_i}\tilde A_i}}_{L^2(A(1/2\tilde R_i,1/\tilde R_i))}^{\dvol_\orb}\\
  &\approx \tilde R_i^{2n-2}\Mod{F_{e^{i\tilde s_i}\tilde A_i}}_{L^2(A(1/2\tilde R_i,1/\tilde R_i))}^{\Euc}\\
 &=\tilde R_i^{2n-1}\Mod{\sig_R^*F_{e^{i\tilde s_i}\tilde
      A_i}}_{L^2(A(1/2,1))}^{\Euc}\leq c\tilde R_i^{-n}=cR_i^{-1}.
\end{align*}
 The second equality in the above estimate is based on the
fact that the Euclidean metric on $\P(1,n)$ lifts to the metric
$\dvol_\orb$ on $\tilde \P$. For the third equality, we observe that
the metric $\dvol_\orb$ is equal to the metric $\tilde
R_i^{-2n+2}\dvol_{\Euc}$ multiplied by a scalar function that takes
values between $1$ and $2^{2n-2}$. We remark that if the annulus were
replaced by a ball we would not have an upper bound for that
multiplicative factor. Finally, the inequality is the result of
applying \eqref{eq:transfcurvbd}, concluding the proof of Proposition
\ref{prop:almostvortexnear0}.
\end{proof}

\section{From vortices to holomorphic maps}
\label{sec:last}

In this section, we prove Theorem \ref{thm:main}
\eqref{part:mainthm2}. The main tool used is the following result on
the asymptotic behavior of a finite energy vortex $(A,u)$ on $\C$,
which is a slight generalization of a result of Ziltener
\cite{Zilt:invaction}.

\begin{proposition}{\rm (Exponential Decay for Vortices)}\label{prop:asym}
  Suppose the $G$-action on $X^{\ss}$ has finite stabilizers and the
  quotient $X \qu G$ is compact. Let $\fix$ be a positive integer such
  that for any $x \in X^{\ss}$, the order of the stabilizer group
  $|G_x|$ divides $\fix$. Let $(A,u)$ be a finite energy vortex on
  $\C\bs B_1$ with target $X$ whose image is contained in a compact subset of
  $X$. Then, for every $\eps>0$, there is a constant $C$ such that
  \begin{align}\label{eq:asymdecay}
    |F_A(z)|^2+|\d_Au(z)|^2+|\Phi(u(z))|^2 \leq C|z|^{-2-\frac 2 n
      +\eps}, \quad \forall z \in \C \bs B_1.
  \end{align}
  The norms are taken with respect to the standard Euclidean metric on
  $\C$.
\end{proposition}

Ziltener \cite{Zilt:invaction} proves this result for $n=1$, in
Section \ref{sec:orbdecay}, we explain how it generalizes to the case
when $X \qu G$ is an orbifold. The following is a conclusion of
Proposition \ref{prop:asym}.  The proofs appear in Section
\ref{sec:orbdecay}.

\begin{proposition}{\rm (Removal of singularity for vortices at infinity)}\label{prop:infvortsing} 
  Assume the setting of Proposition \ref{prop:asym}. Further, assume that the restriction of $A$ in radial gauge to the circle
  $\{|z|=r\}\simeq S^1$ is $\onD+a\d \theta$, for $r \geq 1$.
For any $p>2$ that satisfies $p< \frac 2 {1-\frac 1 n}$ in case $n>1$ and $0 < \gamma < \frac 1 n -1 + \frac 2 p$,
  there exist a constant $c$, a point $x_0 \in \Phi^{-1}(0)$ and a gauge transformation $k_0 \in
  W^{1,p}([0,2\pi],K)$ such that
  \begin{equation}\label{eq:phiconndecay}
    \lim_{r \to \infty}\max_{\theta \in [0,2\pi]}d(x_0,k_0(\theta)u(re^{i\theta}))=0,\quad
    \Mod{k_0^{-1}\partial_\theta k_0+a(r,\cdot)}_{L^p([0,2\pi],K)}<cr^{-\gamma}.
  \end{equation}
  The gauge transformation $k_0$ can be chosen so that $k_0(0)=\Id$. The point $x_0$ is fixed by
  $k_0(2\pi)$, so $k_0(2\pi)^n=\Id$.
\end{proposition}

\subsection{Proof of Theorem \ref{thm:main} \eqref{part:mainthm2}}

\begin{proof}{\rm (Proof of Theorem \ref{thm:main} \eqref{part:mainthm2})}
Let $(A,u)$ be a finite energy affine
vortex with bounded image. The singularity of the vortex $(A,u)$ at infinity can be removed in a weak sense via a unitary gauge transformation of class
  $W^{1,p}$. We work on the chart of the weighted projective line that contains
  the point at infinity, which is $\tU_1 \simeq \C$ in the terminology
  of \eqref{eq:charts}. 
The
  vortex $(A,u)$ lifts to $(\tA,\tu)(z) := (A,u)(z^{-\fix})$,
  defined over $\tU_1 \bs \{0\}$. We assume that $A$, and hence $\tilde A$ is in radial gauge in a punctured neighborhood of $z=0$ in $\tU_1$. We define a gauge
  transformation $\tk_1$ on $\tU_1 \bs \{0\}$ so that the pair $(\tk_1\tA,\tk_1 \tu)$ extends to a gauged holomorphic map in $(L^p_{\loc} \times W^{1,p}_\loc)(\tU_1)$. Define $\tk_1(r,\theta):=k_0(n\theta)$ where $k_0 \in W^{1,p}([0,2\pi],K)$ is given by
  Proposition \ref{prop:infvortsing}. Recall that $k_0$ was only
  defined on $[0,2\pi]$, which can be extended as
  $k_0(\theta+2\pi)=k_0(2\pi)k_0(\theta)$. This way, $\tk_1$ is a
  well-defined gauge transformation on $\C \bs \{0\}$ because
  $k_0(2\pi)^n=\Id$. By Proposition \ref{prop:asym}, $\tk_1\tu$ extends
  continuously over $\infty$, with $\tk_1\tu(\infty)=x_0$. The connection
  $\tk_1\tA$ is in radial gauge, so $\tA=\onD+a\d \theta$. Then, using \eqref{eq:phiconndecay},
  $$\Mod{\tk_1\tA}^p_{L^p(B_1)}= \int_0^1\int_0^{2\pi}|\frac 1 r
  (k_0^{-1}\partial_\theta k_0+a(r^{-n},\theta))|^p r\d \theta dr \leq
  c\int_0^1r^{n\gamma p+1-p}dr.$$
  This expression is finite because $2<p<2(1+\frac 1 n)$, so we can
  choose
  $$\frac 1 n(1-\frac 2 p)<\gamma<\frac 1 n -1 +\frac 2 p .$$
  Proposition \ref{prop:infvortsing} is applicable with these
  values.
  
  If the gauge transformation $k_0$ were smooth, we could have used it
  to construct a principal bundle over the weighted projective line,
  on which the given vortex would be $p$-bounded. But since
  $k_0$ is not smooth, we have to take an indirect approach involving
  complex gauge transformations. We transform the vortex into a smooth
  gauged holomorphic map on a $K$-bundle over the weighted projective line. Then, we
  argue that the complex gauge transformation used is $p$-bounded over that bundle.

We now show that there is a complex gauge transformation, called $\tilde h$, in a punctured neighborhood of infinity in the weighted projective line, that takes the given vortex to an untwisted holomorphic map in the neighborhood of infinity, i.e. a trivial connection and holomorphic map.
For $0<\tilde
  R<1$, let $\sig_{\tilde R}:B_1 \to B_{\tilde R}$ denote the dilation
  function $x \mapsto \tilde Rx$. Writing $\tk_1\tA=\onD+\tilde a$, we
  have $\sig_{\tilde R}^*(\tk_1\tA)=\onD+R\tilde a(R\cdot)$. Then,
  $$\Mod{\sig_R^*(\tk_1\tA)}_{L^p(B_1)}=\Mod{\tilde R \tilde a(\tilde
    R\cdot)}_{L^p(B_1)}=\tilde R^{1-\frac 2 p}\Mod{\tilde
    a}_{L^p(B_{\tilde R})}.$$
  For some $\tilde R \in (0,1]$, $\Mod{\sig_{\tilde
      R}^*\tA}_{L^p(B_1)}$ is small enough that Lemma \ref{lem:toflat}
  is applicable. By Lemma \ref{lem:toflat} there is a
  unique $\tilde \xi \in W_0^{1,p}(B_{\tilde R},\k)$ such that the
  connection $e^{i\tilde \xi}\tk_1\tA$ is flat. Further, up to left
  multiplication by a constant, there is a unique gauge transformation
  $\tk_2 \in W^{1,p}(B_{\tilde R},K)$ such that the connection $\tk_2
  e^{i\tilde \xi}\tk_1\tA$ is the trivial connection. Choose $\tk_2$ so
  that $\tk_2(0)=\Id$. Denote by $\tilde h:=\tk_2 e^{i\tilde \xi}\tk_1:B_{\tilde R}\bs \{0\} \to
  G$ the complex gauge transformation we
  have calculated so far. It transforms the map $\tu$ to a holomorphic
  map near infinity. This is because $\delbar (\tilde h \tilde
  u)=0$ on $B_{\tilde R} \bs \{0\}$ and $\tilde h \tilde u(B_{\tilde
    R})$ extends continuously over $0$. So, we can work on a chart of
  $X$ and use the removal of singularities for holomorphic functions
  (\cite[Theorem 3.1]{Stein_Shakarchi}) to show that $\tilde h \tu$
  extends holomorphically over $0$.  Finally, the complex gauge transformation $\tilde h$ satisfies the symmetry property
$\sig_n^*(\tilde h)=k_0(2\pi)\tilde h$.
This follows from the symmetry properties of $\tk_1 \tilde A$, $\tilde \xi$ and $\tilde k_2$. In particular, we have $\sig_n^*(\tk_1 A)=k_0(2\pi)A$. By the uniqueness of
  $\tilde \xi$ and $\tilde k_2$, we can say that $\tilde \xi \circ
  \sig_n=\Ad_{k_0(2\pi)}\tilde \xi$ and $\tk_2 \circ
  \sig_n=k_0(2\pi)\tk_2k_0(2\pi)^{-1}$. 

  The 
  complex gauge transformation $\tilde h$ calculated in the last paragraph will extend over the affine line after
  it is given a constant twist, and it will transform the given vortex to a gauged holomorphic map that extends smoothly over infinity.  
Suppose $k_0:[0,2\pi] \to K$ is homotopic (with fixed end
  points) to the geodesic $\theta \mapsto e^{-\lambda \theta}$. The
  complex gauge transformation $e^{n\lambda \theta}\tilde h$ is symmetric under the $\Z_n$-action, so it
  descends to a complex gauge transformation $g_1$ on $\C \bs B_R$ that
  satisfies 
$e^{n\lambda \theta}\tilde h(z)=g_1(\frac 1 {z^n})$. 
On $\C \bs B_R$,  $g_1 A$ is equal to $\onD+\lambda \d
  \theta$, which is a smooth connection. The map $g_1u$ is $\delbar_{g_1A}$-holomorphic, so it is smooth on $\C \bs B_R$ by elliptic regularity. 
The complex gauge transformation $g_1:\C \bs B_R \to
  G$ is smooth because both $u$ and $g_1u$ are smooth and $u(\C \bs B_R)
  \subset X^{\ss}$, where $G$ acts with finite stabilizers.
The map $g_1:\C \bs B_R\to G$ is smoothly homotopic to
  identity.  This is because on $B_{\tilde R}\bs \{0\} \subset \tU_1$,
  both $\tk_1^{-1}\tk_2e^{i\tilde \xi} \tilde k_1$ and $e^{n\lambda
    \theta}\tk_1$ are continuously homotopic to identity. Therefore, 
there is a smooth complex gauge transformation $g$ on $\C$ that agrees with $g_1$ on $\C \bs B_{2R}$.
  The gauged
  holomorphic map $g(A,u)$ satisfies the conditions in Proposition
  \ref{prop:ghext} - i.e. on $\C \bs B_{2R}$, $g A =\onD+\lambda \d
  \theta$ and $\lim_{r \to \infty}e^{-\lambda \theta}g
  u(r,\theta)=x_0$. So, $g(A,u)$ extends to a smooth gauged holomorphic map
  over $P \to \P(1,n)$ where the principal $K$-bundle $P$ is described by
  transition functions $\mu=e^{-2\pi\lambda}$ and $\tau=e^{n\lambda
    \theta}$.
  This indeed proves the existence statement of the theorem: the
  complex gauge transformation $g$ can be written as $g=e^{i\zeta}k$,
  where $k \in \K(\C)$ and $\zeta: \C \to \k$ are smooth. We claim
  that $e^{-i\zeta} \in \G(P)_{\on{bd}}$, so that $k(A,u)$ extends to
  a $p$-bounded gauged holomorphic map over $P \to \P(1,n)$. Let
  $\tilde \zeta:\tU_1\bs \{0\} \to \k$ be defined as $\tilde
  \zeta(z)=\zeta(\frac 1 {z^n})$ and let $\tk:\tU_1\bs \{0\} \to K$ be
  equal to $\tk(z)=k(\frac 1 {z^n})$. By straightforward calculations,
  on $B_{\tilde R/2^n}\bs \{0\} \subset \tU_1\bs \{0\}$, we have
$$e^{i\tilde \zeta}=e^{n\lambda \theta}\tk_2 e^{i\tilde \xi}\tk_2^{-1}e^{-n\lambda \theta}, \quad \tk=e^{n\lambda \theta}\tk_2\tk_1.$$
In the trivialization of $P$ on $\tU_1$, the complex gauge
transformation $e^{i\zeta}$ is given by $e^{-n\lambda
  \theta}e^{i\tilde \zeta}e^{n\lambda \theta}:\tU_1 \bs \{0\}\to K$,
which is equal to $\tk_2 e^{-i\tilde \xi} \tk_2^{-1}$ on $B_{\tilde R/2^n}\bs
\{0\} \subset \tU_1\bs \{0\}$. By recalling details from the previous
paragraph, $\tk_2 e^{-i\tilde \xi} \tk_2^{-1}$ extends to a $W^{1,p}$ map on
$B_{\tilde R/2^n}$. Therefore $g$ is in $\G(P)_{\on{bd}}$ and the pair
$k(A,u)$ extends to a $p$-bounded gauged holomorphic map over
$P$. Hence, the pair $(A,u)$ extends to a $p$-bounded gauged
holomorphic map over a bundle that is isomorphic to $P$ (see
\eqref{eq:orbibundleiso}).

  Finally we show uniqueness.  Given a finite energy vortex $(A,u)$,
  the bundle $P \to \P(1,n)$ is determined uniquely by the path
  $\{ [0,\frac {2\pi} n] \ni \theta \mapsto k_0(\theta)\}$ up to
  homotopy and the action of $\Ad_k$ for $k \in K$ (see Remark
  \ref{rem:orbibundleinv}). The path $k_0$ is determined using
  Proposition \ref{prop:infvortsing}, and the choice of this equivalence
  class is unique because it has to satisfy the condition
$$\lim_{r \to \infty}\max_{\theta \in
  [0,2\pi]}d(x_0,k_0(\theta)u(re^{i\theta}))=0,\quad \text{where }
x_0=\lim_{r \to \infty} u(r,0) \in \Phi^{-1}(0).$$
Now suppose $g_1, g_2 \in \G(P)_{\on{bd}}$ so that $g_i(A,u)$ is a
smooth gauged holomorphic map on $P$. On $B_{\tilde R} \subset \tU_1$,
$g_1u$, $g_2u$ are smooth maps to $X^{\ss}$, where $G$ acts with finite stabilizers. So, $g_1^{-1}g_2$ is smooth in this region and hence
$g_1^{-1}g_2 \in \G(P)$. This finishes the proof of Theorem
\ref{thm:main} \eqref{part:mainthm2}.
\end{proof}
The following Lemma is used in the proof of Theorem \ref{thm:main}
\eqref{part:mainthm2}. It says that on a trivial principal bundle, an
$L^p$-small connection can be transformed to a flat connection via a
$W^{1,p}$-small complex gauge transformation.
\begin{lemma}
  \label{lem:toflat} Let  $p>2$ and $\Sig$ be a compact connected Riemann surface with metric with non-empty boundary. Let $P:=\Sig \times K$ be the trivial principal $K$-bundle
  on $\Sig$. Let $\onD$ be the trivial connection on $P$. There are
  constants $c_1$, $c_2$ and $c_2'$ so that the following holds. Let
  $A=\onD +a$ be a connection on $P$ so that $a \in
  \Om^1(\Sig,\k)_{L^p}$. If $\Mod{a}_{L^p(\Sig)}<c_1$, there is a
  unique $\xi \in W^{1,p}(\Sig,\k)$ satisfying $\xi|_{\partial \Sig}=0$,
  $F_{e^{i\xi}A}=0$ and $\Mod{\xi}_{W^{1,p}} \leq
  c_2\Mod{F_A}_{W^{-1,p}} \leq c_2'\Mod{a}_{L^p}$.
  
  Further, on any contractible closed set $\Sig' \subset \on{int}
  \Sig$, there is a gauge transformation $k \in W^{1,p}(\Sig',K)$ so
  that $ke^{i\xi}A=\onD$ on $\Sig'$.  The gauge transformation $k$ is
  unique up to left multiplication by a constant element in $K$.
\end{lemma}

On a compact connected Riemann surface $\Sig$ with non-empty boundary, if the Lie group
$K$ is connected, then any principal $K$-bundle $P \to \Sig$ is
trivializable.  Hence it suffices to consider only trivial bundles in
Lemma \ref{lem:toflat}.

\begin{proof}{\rm (Proof of Lemma \ref{lem:toflat})} The proof is by an
  application of the implicit function theorem on the map
  $$ \F^A:W_0^{1,p}(\Sig,\k) \to W^{-1,p}(\Sig,\k), \quad \xi \mapsto
  *F_{e^{i\xi}A} .$$
  By \eqref{eq:fdepa} and \eqref{eq:infgcaction}, the linearization at
  $\xi$ is
  \begin{align}\label{eq:weaklap}
    D\F^A_\xi=\d_{e^{i\xi}A}^*\d_{e^{i\xi}A}:W_0^{1,p}(\Sig,\k) &\to
    W^{-1,p}(\Sig,\k).
  \end{align}
  In this proof, the constant $c$ indicates a constant that is independent of the
  connection $A$, whose value varies across appearances.

 \vskip.1in \noindent {\sc Step 1:} {\em There is a constant $c_1$ such
    that if a connection $A=\onD+a$ satisfies $\Mod{a}_{L^p}<c_1$,
    then $D\F^A(0)$ is invertible and the inverse
    has norm independent of $a$.}  \\
  The operator $\d^* \d:W_0^{1,p}(\Sig,\k)\to W^{-1,p}(\Sig,\k)$ is an
  isomorphism (see appendix D in \cite{Weh:Uh}). Further,
  \begin{align}\label{eq:lapdiff}
    (\d_A^*\d_A-\d^* \d)\eta=*[a \wedge *\d\eta]+\d^*[a, \eta]+*[a
    \wedge *[a\wedge\eta]].
  \end{align}
  Using Sobolev multiplication (Proposition \ref{prop:sobmult}) and by assuming $c_1<1$, we have 
  \begin{align*}
    \Mod{(\d_A^*\d_A-\d^* \d)\eta}_{W^{-1,p}} \leq
    c\Mod{\eta}_{W^{1,p}}(\Mod{a}_{L^p} + \Mod{a}_{L^p}^2) \leq  c\Mod{\eta}_{W^{1,p}}\Mod{a}_{L^p}
  \end{align*}
  so the operator norm satisfies $\Mod{(\d_A^*\d_A-\d^* \d)} \leq
  c\Mod{a}_{L^p}$. If this quantity is small enough, by a Neumann
  series argument, $\d_A^*\d_A$ is invertible with bounded norm. Let
  $\Theta:=(\d_A^*\d_A-\d^* \d)(\d^*\d)^{-1}$.  If $\Mod{\Theta } < 1$,
  then the operator $\Id + \Theta$ is invertible and the inverse has a
  bound
  $$\Mod{(\Id + \Theta)^{-1}} \leq \sum_{i=0}^\infty \Mod{\Theta}^i.$$
  Further,
  $$(\d_A^*\d_A)^{-1}=(\d^*\d)^{-1}(\Id + \Theta)^{-1}.$$
  We can choose $c_1$ such that if $\Mod{a}_{L^p}<c_1$, then,
  $$\Mod{(\d_A^*\d_A)^{-1}} \leq 2\Mod{(\d^* \d)^{-1}}.$$
  In the statement of the implicit function theorem (Proposition
  \ref{prop:impfnMS}), we take $C:=2\Mod{(\d^* \d)^{-1}}$.

  \vskip .1in \noindent {\sc Step 2:} {\em Given a constant $c_1$, there is a constant $c$
    such that for any connection $A=\onD+a$ satisfying
    $\Mod{a}_{L^p}<c_1$ and $\Mod{\xi}_{W^{1,p}}<1$,
  $$\Mod{D\F^A_\xi-D\F^A_0}<c\Mod{\xi}_{W^{1,p}}.$$}

Write
  $$D\F^A_\xi-D\F^A_0=(\d_{(\exp i\xi)A}^*\d_{(\exp i\xi)A}-\d_A^*\d_A)$$
  Let $(\exp i\xi)A=A+\alpha$. Then,
  $$(\d_{(\exp i\xi)A}^*\d_{(\exp i\xi)A}-\d_A^*\d_A)\xi_1=*[\alpha \wedge
  *\d_A\xi_1]+\d_A^*[\alpha \wedge \xi_1]+*[\alpha \wedge *[\alpha
  \wedge\xi_1]].$$ 
As in Step 1, the operator norm of $D\F_\xi-D\F_0$
  is bounded as
$$\Mod{\d_{(\exp i\xi)A}^*\d_{(\exp i\xi)A}-\d_A^*\d_A} \leq
c(\Mod{\alpha}_{L^p}+\Mod{\alpha}_{L^p}^2) \leq c\Mod{\xi}_{W^{1,p}},$$
where the last inequality is by Lemma \ref{lem:gcactona} and the fact that $\Mod{\xi}_{W^{1,p}}<1$. If
$\Mod{a}_{L^p}<c_1$, the constant $c$ is independent of $a$ and $\xi$.

\vskip .1in \noindent {\sc Step 3:} {\em Finishing the proof.}\\
By the result in Step 2, there is a constant $\delta_0>0$ such that
  $$\Mod \xi_{W^{1,p}} < \delta_0 \implies \Mod{D\F^A_\xi-D\F^A_0} \leq
  \frac 1 {2C} .$$
  The function
  $$(\xi,\Delta \xi) \mapsto D\F^A(\xi)\Delta \xi:W^{1,p}_0(\Sig,\k) \times W^{1,p}_0(\Sig,\k) \to W^{-1,p}(\Sig,\k)$$
  is continuous using the identity
  \eqref{eq:lapdiff}, Lemma \ref{lem:gcactona} and Sobolev
  multiplication (Proposition \ref{prop:sobmult}). So $\F^A$ is
  differentiable and the implicit function theorem is applicable.
  The constant $c_1$ can be chosen such that $\Mod{F_A}_{-1,p}<\frac {\delta_0}
  {8C}$. This is because by the multiplication theorem,
  $$\Mod{F_A}_{-1,p} \leq c(\Mod{a}_{L^p}+\Mod{a}_{L^p}^2)<cc_1.$$
  Finally, apply Proposition \ref{prop:impfnMS} with
  $\delta:=8C\Mod{F_A}_{-1,p}$. This ensures $\Mod{\F^A(0)}=\Mod{F_A}_{-1,p}<\frac
  \delta {4C}$ and we get a solution $\xi \in B_\delta(W^{1,p})$ for
  $\F^A(\xi)=0$. The solution $\xi$ also satisfies
  $\Mod{\xi}_{1,p}<8C\Mod{F(A)}_{-1,p}<c\Mod{a}_{L^p}$. The constants in this inequality are suitable values for $c_2$ and $c_2'$ in the statement of the Lemma.

  We next prove the second statement of the lemma.  We use the fact that a
  flat $L^p$ connection is gauge-equivalent to a connection that is smooth
  away from the boundary $\partial \Sig$.  Indeed, this is proved in
  \cite[Theorem 3.1]{Weh:weakflat} for a base manifold without
  boundary, but since we only want interior regularity, the proof is
  identical. We briefly outline this proof: For an $L^p$ connection
  $A$, if we choose a smooth connection $\tilde A_0$ close enough to $A$
  (in the $L^p$ norm), there is a gauge transformation $k \in W^{1,p}$
  that puts $A$ in coulomb gauge with respect to $\tilde A_0$, that
  is, if $a:=kA-\tilde A_0$,
  $$\d_{\tilde A_0}^*a=0 .$$
  This is the content of the local slice theorem (Theorem 8.3) in
  \cite{Weh:Uh}. Now, one has control over $\d a$ and $\d^*a$, using
  elliptic bootstrapping as in \cite{Weh:weakflat}, it can be shown
  that $kA$ is smooth away from $\partial \Sig$. Similar to the first
  statement in the lemma, here we rely on the ellipticity of
  $\Delta:W^{1,p}_0 \to W^{-1,p}$.

  Finally note that a smooth flat connection on a contractible set is
  gauge equivalent to the trivial connection. The uniqueness statement
  in the lemma is obvious because the trivial connection is preserved
  only by constant gauge transformations. This finishes the proof of
  Lemma \ref{lem:toflat}
\end{proof}
\begin{remark}\label{rem:toflat_higher} Lemma \ref{lem:toflat} is true
  for higher regularity connections. Suppose $k \in \Z_{\geq 0}$,
  $p>1$ such that $(k+1)p>2$.  If $a \in W^{k,p}$, then $\xi \in
  W^{k+1,p}$ can be produced satisfying the conditions of Lemma
  \ref{lem:toflat}. The proof of
  the higher regularity statement is analogous and is more standard.
\end{remark}

\subsection{Unique affine vortex in a complex gauge orbit}

The following is the last component in the proof of the main Theorem
\ref{thm:main}. To state it, we assume the result of part
\eqref{part:mainthm2} of this theorem.
\begin{proposition} {\rm (At most one vortex in a complex gauge
    orbit)} \label{prop:uniqueaffinevort} Suppose $(A_0,u_0)$ and
  $(A_1,u_1)$ are finite energy vortices on $\C$ that extend to a
  $p$-bounded gauged holomorphic map over a bundle $P \to
  \P(1,n)$. Suppose further that they are related by a complex gauge
  transformation $g \in \G(P)_{\on{bd}}$. Then $g$ is a unitary gauge
  transformation, i.e. $g \in \K(P)_{\on{bd}}$.
\end{proposition}
\begin{proof} The proof proceeds in the same way as the corresponding
  result for vortices on a compact Riemann surface (Proposition
  \ref{prop:atmost1vortex}).  The complex gauge transformation $g$ can
  be written as $g=ke^{i\xi}$, where $k \in \K(P)_{\on{bd}}$ and $\xi
  \in \Gamma(\P(1,n),P(\k))_{\on{bd}}$. We can assume $k=\Id$ and
  $(A_1,u_1)=e^{i\xi}(A_0,u_0)$. Let $(A_t,u_t):=e^{it\xi}(A,u)$ for
  $t \in [0,1]$. We write
  \begin{equation}\label{eq:bypartsaffine}
    \begin{split}
      \ddt \int_\C\lan *F_{A_t,u_t},\xi \ran &= \int_\C\langle \d_{A_t}^*\d_{A_t}\xi+u_t^*d\Phi(J(\xi)_{u_t}),\xi\rangle_\k\\
      &=\Mod{\d_{A_t}\xi}_{L^2}^2 +
      \int_\C\omega_{u_t}((\xi)_{u_t},J(\xi)_{u_t}) + \lim_{r \to
        \infty}\int_{\partial B_r}\lan \nabla_{A_t,\nu}\xi,\xi\ran_\k
    \end{split}
  \end{equation}
  where $\nabla_{A_t,\nu}\xi$, is the covariant derivative of $\xi$
  along $\nu$, the outward normal unit vector field to $\partial B_r$.

  However, for the above computation to make sense, we need to show
  that the terms $\Mod{\d_{A_t}\xi}_{L^2(\C)}$ and
  $\int_\C\omega_{u_t}((\xi)_{u_t},J(\xi)_{u_t})$ are finite and the
  boundary term vanishes.
  
  \vskip .1in \noindent {\sc Step 1}: {\em $\Mod{\d_{A_t}\xi}_{L^2}<\infty$ for $t \in
    [0,1]$.}\\
  We work on the chart containing infinity - this is $\tU_1$ from
  \eqref{eq:charts}. Consider $B_{\tilde R} \subset \tU_1$, for some
  $\tilde R>0$. We have $A_0 \in L^p(B_{\tilde R})$ and $\xi \in
  W^{1,p}(B_{\tilde R})$. Using Lemma \ref{lem:gcactona} for the
  action of the complex gauge transformation $e^{it\xi}$ on the
  connection $A_0$, we can say $A_t-A_0 \in L^p(B_{\tilde R})$. Since
  $p>2$, $\Mod{\d_{A_t}\xi}_{L^2(B_{\tilde R})}$ is finite. Let
  $R=(\tilde R)^{-1/n}$. By conformal invariance of the $L^2$ norm of
  one-forms, the norms $\Mod{\d_{A_t}\xi}_{L^2(\C \bs B_R)}$ and
  $\Mod{\d_{A_t}\xi}_{L^2(B_{\tilde R})}$ are equal, so both are
  finite, and hence $\Mod{\d_{A_t}\xi}_{L^2}<\infty$ for $0\leq t \leq
  1$.

  \vskip .1in \noindent {\sc Step 2}: {\em
    $\int_\C\omega_{u_t}((\xi)_{u_t},J(\xi)_{u_t})<\infty$ for $t \in
    [0,1]$.}\\
  We use asymptotic decay of vortices (Proposition \ref{prop:asym}) to
  obtain an asymptotic bound on $|\xi|$: fix an $0<\eps<\frac 2 n$,
  let $\delta=\frac 2 n -\eps$, then for an affine vortex $(A,u)$
  there is a constant $c$ so that for $z \in \C \bs B_1$,
  $$e_{A,u}(z) \leq c|z|^{-2-\delta} .$$
  Therefore, $|\Phi(u_0(z))|$, $|\Phi(u_1(z))| \leq c|z|^{-1-\frac
    \delta 2}$.

  We prove a similar asymptotic bound on $\xi$ also. For any $c>0$,
  define $Z_c:=\{|\Phi|<c\} \subset X$. For any $x \in X$, define
  $$\Psi_x: \k \to \k, \quad s \mapsto \Phi(e^{is}x) - \Phi(x).$$
  For all $x \in \Phi^{-1}(0)$, the derivative $d\Psi_x(0)$ is
  invertible. Since $\Phi$ is proper, we can choose $\eps_1>0$ such
  that $d\Psi_x(0)$ is invertible for all $x \in Z_{\eps_1}$. Define
  \begin{equation*}
    \Psi:Z_{\eps_1} \times  \k \to Z_{\eps_1} \times \k \quad (z,s) \mapsto (z,\Psi_z(s))=(z, \Phi(e^{is}z)-\Phi(z)).
  \end{equation*}
  For all $z \in Z_{\eps_1}$, the map $d\Psi(z,0)$ is an isomorphism,
  so $\Psi$ is locally an isomorphism. That is, there is an open
  neighborhood of $U$ of $(z,0)$, such that $\Psi:U \to f(U)$ is a
  diffeomorphism. By the compactness of $\ol Z_{\eps_1}$, 
  there exists $\eps_2>0$ such that $\Psi$ has a smooth inverse on
  $Z_{\eps_1} \times B_{\eps_2}(\k)$. Set
  $\eps:=\min\{\eps_1,\eps_2\}$. We can conclude that there is a
  constant $c$ such that for $x$, $y \in Z_\eps$ satisfying
  $y=e^{is}x$ for some $s \in \k$,
  \begin{equation}\label{eq:sPhicont}
    |s| < c|\Phi(x)-\Phi(y)|.
  \end{equation}
  We can choose $0<\eps_2'<\eps_2$ such that for all $z \in
  Z_{\eps_1}$, the straight line connecting $0$ and any point in
  $\Psi_z^{-1}(B_{\eps_2'})$ is contained in
  $\Psi_z^{-1}(B_{\eps_2})$. Set $\eps':=\min\{\eps_1,\eps_2'\}$. For
  any $x$, $y \in Z_{\eps'}$ satisfying $y=e^{is}x$ for some $s \in
  \k$, the curve $\{e^{its}x:t \in [0,1]\}$ is contained in $Z_\eps$.

  Let $R>0$ be large enough that both $u_0$, $u_1$ map $\C \bs B_R$ to
  $Z_{\eps'}$. Recall $u_1=e^{i\xi}u_0$. By the preceding discussion,
  $u_t(\C \bs B_R) \subset Z_\eps$ for all $t \in [0,1]$ and by
  \eqref{eq:sPhicont},
  \begin{equation}\label{eq:xibd}
    |\xi(x)| < c|\Phi(u_1(x))-\Phi(u_0(x))| < c|z|^{-1-\frac \delta 2}
  \end{equation}
  for $x \in \C \bs B_R$. For $z \in Z_\eps$, the operator $L_z$ (see
  \eqref{eq:defL}) is uniformly bounded. So, there is a constant $c$
  so that
  $$|\omega_{u_t}((\xi)_{u_t},J(\xi)_{u_t})(z)|<c|z|^{-2-\delta}$$
  for $z \in \C \bs B_R$ and hence,
  $$\int_\C\omega_{u_t}((\xi)_{u_t},J(\xi)_{u_t})<\infty, \quad 0\leq t \leq 1 .$$

  \vskip .1in \noindent {\sc Step 3}: {\em $\lim_{r \to \infty}\int_{\partial B_r}\lan
    \nabla_{A_t,\nu}\xi,\xi\ran_\k=0$ for $t \in [0,1]$.}\\ In step 1,
  we showed that $A_t$ is in $L^2(\C \bs B_R)$, but we do not know
  anything yet about the restriction of $A_t$ to the boundary
  $\partial B_R$. We now obtain a $C^0$-bound under suitable local
  trivializations of $P$. For this we cover $\C \bs B_R$ by identical
  open sets. Let $S \subseteq \C$ be an open set with smooth boundary
  such that
  $$[-\tqq,\tqq]\times [-\tqq,\tqq] \subseteq S \subseteq [-1,1]
  \times [-1,1] .$$
  Then, $\{S+(x,y):|x|,|y| \geq R-2\}$ is a cover of $\C \bs B_R$. Let
  $S''\Subset S' \Subset S$ be such that their translates (by
  integers) also cover $\C \bs B_R$.  The quantity
  $\lan\d_{A_t}\xi,\xi\ran_\k \in \Om^1(\Sigma)$ is gauge-invariant,
  so on each $S+(x,y)$, we can choose a different trivialization to
  study it. We focus on a single set $S_{xy}:=S+(x,y)$ and let
  $r:=\sqrt{x^2+y^2}$. In the following discussion, the constant $c$
  is independent of $(x,y)$ and $r$. Fix a trivialization of
  $P|_{S_{xy}}$ so that $A_0=\onD+a_0$ is in Uhlenbeck gauge i.e.
  \begin{align}\label{eq:a0bd}
    \Mod{a_0}_{W^{1,p}(S_{xy})} < c\Mod{F(A_0)}_{L^p(S_{xy})}<c.
  \end{align}
  Under this trivialization, we write $A_t=\onD+a_t$. By applying a
  gauge transformation $k:S_{xy}\to K$, we can put $A_1$ in Uhlenbeck
  gauge - i.e. if $kA_1=\onD+\tilde a_1$ then,
  \begin{align}\label{eq:a1bd}
    \Mod{\tilde a_1}_{W^{1,p}(S_{xy})} <
    c\Mod{F(A_1)}_{L^p(S_{xy})}<c.
  \end{align}
  Denote $g=ke^{i\xi}$, so $gA_0=\onD+\tilde a_1$. As in the proof of
  Lemma \ref{lem:complexgreg}, we can write
  $a^{0,1}g=-\delbar_{A_0}g$, where $a=\tilde a_1-a_0$. By elliptic
  regularity,
  \begin{align*}
    \Mod{g}_{W^{1,p}(S'_{xy})} &\leq c(\Mod{\delbar
      g}_{L^p(S_{xy})}+\Mod g_{L^p(S_{xy})}) \\ & \leq
    c(\Mod{a^{0,1}g}_{L^p(S_{xy})}+\Mod{a_0^{0,1}g}_{L^p(S_{xy})})+c\Mod
    g_{L^\infty(S_{xy})})
  \end{align*}
  There is an $L^\infty$ bound on $g$ from \eqref{eq:xibd} and the
  fact that $K$ is compact. Together with the bounds on $a$ and $a_0$
  (from \eqref{eq:a0bd} and \eqref{eq:a1bd}), this shows that
  $\Mod{g}_{W^{1,p}(S'_{xy})} \leq c$. By applying elliptic regularity
  again, we can show $\Mod{g}_{W^{2,p}(S''_{xy})} \leq c$. Since
  \eqref{eq:complexgroupiso} is a diffeomorphism,
  \begin{align}\label{eq:xiderivativebd}
    \Mod{\xi}_{W^{2,p}(S''_{xy})} \leq c.
  \end{align}
  By Lemma \ref{lem:gcactona}, we get
  $$\Mod{a_t-a_0}_{W^{1,p}(S''_{xy})}<c.$$ 
  Together with the bound on $a_0$, this uniformly bounds
  $\Mod{a_t}_{W^{1,p}(S_{xy})}$. By Sobolev embedding
  \begin{align}\label{eq:atC0bd}
    \Mod{a_t}_{C^0(S_{xy})}<c
  \end{align}
  for $0 \leq t \leq 1$.
 
  Consider the integral $\int_{\partial B_r}\lan
  \nabla_{A_t,\nu}\xi,\xi\ran_\k$. We partition the curve into
  segments, so that each segment lies in a single set
  $S''_{xy}$. Write $\nabla_{A_t,\nu}\xi=\ppnu \xi +
  [(a_t)_\nu,\xi]$. We have $C^0$ bounds on $\ppnu \xi$ and
  $(a_t)_\nu$ using \eqref{eq:xiderivativebd} and
  \eqref{eq:atC0bd}. Together with the asymptotic bound on $\xi$
  \eqref{eq:xibd}, the result of Step 3 is proved.

  \vskip .1in \noindent {\sc Step 4}: {\em Finishing the proof.}\\
  We have shown that the equation \eqref{eq:bypartsaffine} holds and
  the boundary term vanishes as $r \to \infty$. So,
  $$\ddt \int_\C\lan *F_{A_t,u_t},\xi \ran>0$$ 
  if $\xi \neq 0$. Since $F_{A_0,u_0}=F_{A_1,u_1}=0$, this implies
  that $\xi=0$, proving Proposition \ref{prop:uniqueaffinevort}.
\end{proof}

\begin{proof}{\rm (Proof of Proposition \ref{prop:toric})} We first consider
  the case of a line bundle $X=\C$ on which $G$ acts with weight $\nu
  \in \g^\dual$. Suppose $P$ is a principal bundle with $[P]=d$. Then the
  associated orbifold line bundle $P(X)$ has degree (pairing of the
  first Chern class with the rational fundamental class) $\lan \nu,d\ran
  \in \Q$. By our set-up $\lan \nu,d\ran =\frac m n$ for some $m \in
  \Z$. Up to isomorphism, there is a unique holomorphic line bundle on
  $\P(1,n)$ of this degree \cite[Theorem 5.1]{Smith}, this line bundle
  is described as: $L=(\tU_1 \times \C) \sqcup (U_2 \times \C)/\sim$
  where $\sim$ is as follows:
  \begin{align*}
    (z,l) &\sim (\sig_n z, e^{-2\pi im/n}l) \qquad (z,l) \in \tU_1 \times \C\\
    (z,l) &\sim (w,z^m l) \qquad 0 \neq z \in \tU_1, w \in U_2,
    w=\frac 1 {z^\fix}, \in \C
  \end{align*}
  In the framework of Section \ref{sec:orbibundle}, the transition
  functions $\mu:\tU_1 \to \C^\times$ and $\tau:\tU_1 \bs \{0\} \to
  \C^\times$ are holomorphic and are $\mu(z):=e^{-2\pi im/n}$ and
  $\tau(z)=z^m$.  Holomorphic sections on $u:\P(1,n) \to P(X)$ are of
  the following form:
$$U_2 \ni w \mapsto u_0 + u_1w+ \dots + u_qw^q, \quad \tU_1 \ni z
\mapsto z^m(u_0 + u_1 z^{-n} +\dots+ u_q z^{-qn}),$$ where $q=\lfloor
m/n \rfloor$. It is clear that $u(\infty)=0$ if $m/n$ is not an
integer, and otherwise $u(\infty) = u_q=u^{(q)}/q!$. For the case of a
vector bundle $X=\C^k$, the expression for $u_\infty$ can be derived
in a similar way. Using Theorem \ref{thm:main}, the section $u$ is
complex gauge equivalent to an affine vortex if and only if $u(\infty)
\in X^\ss$.

We next show when two such sections are complex gauge
equivalent. 
Consider a holomorphic complex gauge transformation $g \in \G(P)$
relating sections $u_1,u_2$. 
The complex gauge transformation $g$ induces an isomorphism on each line bundle $P \times_{\nu_i}\C$, which we call $\nu_i(g)$. Since the group is Abelian, $\nu_i(g)$ is a holomorphic map from $\P(1,n)$ to $\C^\times$, which is necessarily a constant. Since the weights $\nu_1,\dots,\nu_k$ span $\g^\dual$, we conclude $g:\P(1,n) \to G$ is a constant.
\end{proof}
\section{Asymptotic decay for vortices} \label{sec:orbdecay} 

Proposition \ref{prop:asym} is a consequence of a decay result for
vortices on a cylinder (Proposition \ref{prop:decaycyl}) and a result
about the limit behavior of $u$ as $z \to \infty$ (Proposition
\ref{prop:infusing}).

\begin{definition}{\rm (Energy density)} Let $U \subset \C$ be an open subset. Suppose $\lambda : U
  \to \R_{\geq 0}$ is a smooth function. Let $(A,u)$ be a vortex from
  the principal bundle $P \to U$ to $X$ with respect to the metric
  $\om=\lambda^2 ds \wedge dt$. The {\em energy density} of $(A,u)$ is
  $$e_{(A,u)}(z):=\hh \left(|F_A(z)|^2+|\d_Au(z)|^2+|\Phi(u(z))|^2\right)$$ 
  for any $z \in U$. The norms are defined as follows : for any form $\alpha
  \in \Om^*(U,\k)$, if $\alpha \wedge *_\lambda\alpha= f ds\wedge dt$,
  then $|\alpha(z)|^2=|f|^2$. Here, $*_\lambda$ denotes the Hodge star
  with respect to the metric $\lambda^2 ds\wedge dt$.
\end{definition}

\begin{definition}{\rm (Admissible metric)} Let $a>0$. On a half cylinder
  $\Sig:=\{(s,t):s \geq 0, t \in \R/a\Z\}$, the metric
  $\om_\Sig=\lambda^2ds \wedge dt$ is {\em admissible} if 
  \begin{align}\label{eq:admmetric}
    \lambda \geq \frac {2\pi} {a m_{\Phi^{-1}(0)}} \quad \sup_\Sig(\Delta(\lambda^{-2})+|d(\lambda^{-1})|^2) <2m^2_{\Phi^{-1}(0)}
  \end{align}
  where $m_{\Phi^{-1}(0)}=\inf_{x \in \Phi^{-1}(0)}|L_x|$, where $L_x$ is defined by \eqref{eq:defL}.
\end{definition}

\begin{proposition}\label{prop:decaycyl} {\rm(Decay for vortices on the half cylinder,
    \cite[Theorem 1.3]{Zilt:invaction})} Let
  $\Sig$ be a half cylinder $$\Sig:=\{(s,t):s \geq 0, t \in \R/a\Z\}$$
  with an admissible metric $\om_\Sig=\lambda^2ds \wedge dt$ where
  $\lambda = \lambda(s,t)$ is a positive function.  Suppose $G$ acts
  freely on $X^{\ss}$ and $(A,u)$ is a finite energy vortex from the
  trivial bundle $\Sig \times K$ to $X$ such that $\ol{u(\Sig)}$ is
  compact. Then, for every $\eps>0$, there is a constant $C$ such that
  \begin{align*}
    e_{(A,u)}(s+it) \leq C\lambda^{-2}e^{(-\frac {4\pi}a +\eps)s}.
  \end{align*}
\end{proposition}

\begin{proposition}\label{prop:infusing} Suppose $(A,u)$ is a finite
  energy vortex with bounded image on the half cylinder $\Sig$ (described in Proposition
  \ref{prop:decaycyl}). Let $\pi_G:X^{\ss} \to X \qu G$ denote the
  projection. Then, $\lim_{s \to \infty}\pi_G \circ u(s,t)$ exists. 
\end{proposition}

\begin{proof}
  The proof uses a combination of
  results from Ziltener \cite{Zilt:invaction} and McDuff-Salamon \cite{MS} and extends them to the case of orbifold quotient/target. 
  For the mean value inequality, we use Lemma 3.3 in
  \cite{Zilt:invaction}. The proof of this result works when the action
  of $K$ on $\Phi^{-1}(0)$ has finite stabilizers. In the setting of
  Proposition \ref{prop:decaycyl}, this gives that there exists a number
  $s_0 \geq \hh$ such that for $z=(s,t)$ with $s \geq s_0$, 
  \begin{align*}
    e_{(A,u)}(z) \leq \frac {32} {\pi} E((A,u);B_\hh(z)).
  \end{align*} 
  Since $E((A,u); \Sig)$ is finite, the right side goes to $0$ as $s
  \to \infty$. So, for large enough $s_0$, $\Phi(u(s,t))$ is close
  enough to $0$ and so $u(s,t) \in X^{\ss}$. Furthermore, $u_G:=\pi_G \circ u$
  is well-defined and is a holomorphic map to $X \qu G$. Also,
  \begin{equation}
    \begin{split}
      \ell(u_G(\{s=s_1, t \in \R/a\Z\})) &\leq \int_0^a|du_G(s_1,t)|dt\\ &
      \leq \int_0^a|\d_Au(s_1,t)|dt \to 0 \quad \text{ as } s_1 \to \infty.
  \end{split}
  \end{equation}
  
  Now, we switch to working on $X \qu G$ to prove the result. For
  every $p \in X \qu G$, there is a neighborhood $U_p$ and a
  uniformizing chart $(V_p,G_p,\pi)$ such that $V_p \subset \C^N$,
  $G_p$ is a finite group acting on $U_p$ and $\pi:U_p \to V_p$
  induces a homeomorphism from $U_p/G_p$ to $V_p$ (see
  \cite{ChenRuan}). Each $U_p$ has a $G_p$-invariant symplectic form
  that descends to the symplectic form on $X \qu G$. The quotient $X\qu G$ is
  compact, so it can be covered by a finite number of such
  neighborhoods $U_1, \dots, U_k$. Fix a constant $\delta_0>0$ such
  that for any $p \in X \qu G$, $B_p(\delta_0) \subset U_i$
   for some $i \in
  \{1,\dots,k\}$. If the length of the loop $\gamma: S^1 \to X\qu G$
  is less than $\delta_0$, it can be lifted to the cover in some
  uniformizing chart and the isoperimetric inequality 
  \cite[Theorem 4.4.1]{MS} can be applied. The rest of the proof can be completed in
  the same way as the proof of the removal of singularities result for
  $J$-holomorphic curves in \cite[Theorem 4.1.2]{MS}. We need the
  second paragraph of the proof of \cite[Lemma 4.5.1]{MS} (this requires Stokes'
  theorem for orbifolds, which can be proved by passing to a cover
  locally using regularizing charts), followed by the proof of \cite[Theorem
  4.1.2]{MS}. Note that we don't require holomorphic extension of $u_G$
  over the singularity.
\end{proof}

Proposition \ref{prop:asym} now follows in a straightforward way.

\begin{proof}{\rm (Proof of Proposition \ref{prop:asym})}
  Map $\C \bs B_1$ to the half cylinder $\Sig$, set $a=2\pi$ and apply
  change of coordinates $\Sig \ni z \mapsto e^z \in \C \bs
  B_1$. Equip $\Sig$ with the pullback metric. By Proposition \ref{prop:infusing}, $u_G(\infty):=\lim_{s \to
    \infty}u_G(s,t)$ exists. Let $x \in \pi_G^{-1}(u_G(\infty))$, and
  let $S$ be a slice for the $G$-action at $x$. This means
  $G\times_{G_x}S \to X^{\ss}$ is a diffeomorphism onto its
  image. The map $\pi:G \times S \to X^{\ss}$ is a $|G_x|$-cover, equip $G
  \times S$ with the metric $\pi^*\om_X$. The left $K$-action is
  free and has moment map $\pi^*\Phi$.  The number $n$ divides $G_x$, so for some
  large $s_0$, $u(\Sig_{s>s_0})\subseteq GS$ and the map $u$ lifts to $\tilde u:
  \tSig_{s>s_0} \to G \times S$. Here $\tilde\Sig_{s>s_0}=\{(s,t):s
  \geq s_0, t \in \R/2\pi n\Z\}$ is an $n$-cover of
  $\Sig_{s>s_0}$ equipped with the pullback metric.
Now, Proposition \ref{prop:decaycyl} can be applied
  to the lift $(\tilde A, \tilde u): \tilde \Sig_{s\geq s_0} \to G \times S$,
  and this proves Proposition \ref{prop:asym}.
\end{proof} 

\begin{proof}{\rm (Proof of Proposition \ref{prop:infvortsing})} 
  We work in cylindrical co-ordinates on the complement of a ball. Map $\C \bs B_1$ to the half
  cylinder $\Sig=\{(s,t):s \geq 0, t \in \R/2\pi\Z\}$, with change of
    coordinates $\Sig \ni z \mapsto e^z \in \C \bs B_1$. The Euclidean
    metric on $\C \bs B_1$ pulls back to $\om_\Sig=e^{2s} \d s \wedge \d t$ on
    $\Sig$.  The connection $A$ can be put in radial gauge, so that on
    $\Sig$, $A=\onD+a(s,t) dt$. The proof now is exactly same as the
    proof of Proposition D.7 (B) in \cite{Zilt:thesis}. The only
    asymptotic result used in that proof is the following: for some $\delta>0$, 
    $$\sup_t(|\pps u(s,t)|+e^{\delta s}|\Phi(u(s,t))|) \leq e^{-\delta s}\quad
    \forall s \geq 0.$$
    In our case, by Proposition \ref{prop:asym} we have this result
    for $\delta=\frac 2 n -\eps$. The conclusion of Proposition D.7
    (B) consists of \eqref{eq:phiconndecay} in
    the form:
    \begin{equation*}
      \Mod{k_0^{-1}\partial_\theta k_0+a(r,\cdot)}_{L^p([0,2\pi],K)}<cr^{(-1+\frac 2 p +\frac \delta 2)}.
    \end{equation*}
    To obtain our result, we choose $\eps=2(\frac 1 n - 1 + \frac 2 p
    - \gamma)$.
\end{proof}

\section{Some analytic results}

In this section, we collect some analytic results used in the
proof. The first two results are standard. The following is
Proposition A.3.4 in \cite{MS}.

\begin{proposition} \label{prop:impfnMS} {\rm (Implicit function
    theorem)} Let $F: X \to Y$ be a differentiable map between Banach
  spaces. Suppose that $DF(0)$ is surjective and has a right inverse
  $Q$, with $\Mod{Q} \leq C$. For all $x \in B_\delta$,
  $\Mod{DF(x)-DF(0)}<\frac 1 {2C}$. If $\Mod{F(0)}<\frac \delta {4C}$,
  then $F(x)=0$ has a unique solution in $B_\delta$ satisfying $x \in \im Q$.
\end{proposition} 

\begin{proposition}{\rm (Sobolev multiplication)}\label{prop:sobmult} 
Let $\Sig$ be a compact $n$-dimensional manifold possibly with a smooth boundary. 
\begin{enumerate}
\item {\rm(\cite[Theorem 4.39]{Adams:sobspace})} Given $k \in \Z$ and $p>1$ be such that $kp>n$. Then, $W^{k,p}(\Sig)$ is a Banach algebra with respect to pointwise multiplication. There is a consant $c$ such that for any $f$, $g \in W^{k,p}(\Sig)$
  \begin{equation*}
    \Mod{fg}_{W^{k,p}} \leq c \Mod f_{W^{k,p}}\Mod g_{W^{k,p}}.
  \end{equation*}
\item {\rm(\cite[Theorem 9.5 (3)]{Palais})} Suppose $l \in \Z$, $k_i \in \Z_{\geq 0}$ and $q$, $p_i>1$, where $i=1, 2$. Suppose at least for one $i$, $k_ip_i<n$. Let $l \leq k_1, k_2$ and $l - \frac n q < \sum_{i=1,2}k_i - \frac n {p_i}$. Further, if $l<0$, then we assume $\sum_{i:k_ip_i<n}(\frac n {p_i}-k_i)<n$. Then, there is a consant $c$ such that for any $f \in W^{k_1,p_1}(\Sig)$, $g \in W^{k_2,p_2}(\Sig)$
  \begin{equation*}
    \Mod{fg}_{W^{l,q}} \leq c \Mod f_{W^{k_1,p_1}}\Mod g_{W^{k_2,p_2}}.
  \end{equation*}
\end{enumerate}
\end{proposition}
The following result is Theorem 1 in \cite{Do:bdry} applied to the
case where the base manifold is a Riemann surface.
\begin{theorem}{\rm (Donaldson \cite{Do:bdry}, Theorem 1)}\label{thm:YMbdry} Let $\Sig$ be a
  compact Riemann surface with boundary, and let $A$ be a $H^1$
  connection on the trivial bundle $\Sig \times K$. There is a unique
  $s \in H^2(\Sig,\k)$ satisfying $s|_{\partial \Sig} \equiv 0$ and
  $e^{is}A$ is a flat connection.
\end{theorem} 
\begin{lemma}{\rm (Action of $\G_\C$ on
    $\A$)} \label{lem:gcactona} Let $(\Sig,j_\Sig)$ be a compact
  Riemann surface,
  possibly with a smooth boundary and $P$ be a principal $K$-bundle. Let $k \in \Z_{\geq 0}$,
  $p>1$ be such that $(k+1)p>2$. Complex gauge transformations in
  $\G^{k+1,p}(P)$ act smoothly on the space of connections $\A^{k,p}(P)$.

  Let $A_0$ be a smooth connection on $P$. For
  any $\eps>0$, there is a constant $C$ so that the following is
  satisfied. For any $W^{k,p}$ connection $A=A_0+a$ which satisfies
  $\Mod{a}_{W^{k,p}(\Sig)}<\eps$ and any $\xi \in W^{k+1,p}(\Sig, P(\k))$
  that satisfies $\Mod{\xi}_{W^{k+1,p}}<1$, 
  \begin{equation}
    \label{eq:gcactbd}
    \Mod{(\exp i\xi)A - A}_{W^{k,p}(\Sig)} \leq C\Mod{\xi}_{W^{k+1,p}(\Sig)}.
  \end{equation}
\end{lemma}

\begin{proof} The action of unitary gauge transformations $\K^{k+1,p}$
  is smooth on $\A^{k,p}$ - this is a standard result, see \cite[Lemma A.5]{Weh:Uh}. To prove the
  same for $\G^{k+1,p}$, it is enough to look at complex gauge
  transformations of the form $e^{i\xi}$, where $\xi \in
  W^{k+1,p}(\Sig, P(\k))$. The infinitesimal action of $i\xi$ on a connection $A$
  is $\d_A\xi \circ j_\Sig$. Suppose $a_t:[0,1] \to
  W^{k,p}(\Sig,P(\k))$ is the solution of the equation
  \begin{equation} \label{ode} 
\frac{da_t}{dt}=(\d_{A+a_t}\xi) \circ j_\Sig =(\d_{A_0}\xi+[a_t, \xi]) \circ j_\Sig, \quad a_0=a.\end{equation}
  Then, $(\exp i\xi)A=A_0+a_1$. The map 
  $$W^{k,p}\times W^{k+1,p} \ni (a,\xi) \mapsto (\d_{A_0}\xi + [a,\xi])\circ j_\Sig \in
  W^{k,p}$$
is smooth. So, the equation \eqref{ode} has a
  solution $t \mapsto a_t$ on the time interval $[0,1]$. Further $a_1$
  varies smoothly with $a_0$ and $\xi$. This shows that $\G^{k+1,p}$
  has a smooth action on $\A^{k,p}$.

  Next, we prove the norm bound assuming
  $\Mod{\xi}_{W^{k+1,p}}<1$. Suppose $a_t \neq 0$ for all $t \in
  [0,1]$. Then, 
  \begin{equation}\label{eq:derivativeat}  
    \frac{d}{dt}\Mod{a_t}_{W^{k,p}} \leq
    \Mod{\frac{da_t}{dt}}_{W^{k,p}} \leq
    c(\Mod{\xi}_{W^{k+1,p}}+\Mod{a_t}_{W^{k,p}}\Mod \xi_{W^{k+1,p}}).
  \end{equation}
  We remark that $\frac{d}{dt}\Mod{a_t}_{W^{k,p}}$ is well-defined
  since $t \mapsto a_t$ is differentiable and $a_t$ is not identically
  zero for any $t \in [0,1]$. Since $\Mod{\xi}_{W^{k+1,p}}<1$, we have
  $\frac{d}{dt}\Mod{a_t}_{W^{k,p}} \leq c(1+\Mod{a}_{W^{k,p}})$. Then,
  $$\Mod{a_t}_{W^{k,p}} \leq e^{ct}(1+\Mod{a_0}_{W^{k,p}})-1 \leq
  C(\eps).$$

 Now, we use \eqref{eq:derivativeat} again and write $\frac{d}{dt}\Mod{a_t}_{W^{k,p}} \leq
  C\Mod{\xi}_{W^{k+1,p}}$. This proves the result if $a_t \neq 0$ for
  all $t \in [0,1]$. If that is not the case, we apply the procedure
  on the intervals $[0,t_0]$ and $[t_1,1]$, where $t_0$, $t_1$ are the
  smallest and largest time values for which $\Mod{a_t}\leq \eps$. This
  would produce the same bound.
\end{proof}

\begin{lemma}\label{lem:hausquot}{\rm(Converging connections are related by converging gauge transformations)} Let $\Sig \subset \C$ be the closure of a bounded open subset with smooth boundary and let $P:=\Sig \times K$ be the trivial principal bundle over $\Sig$. Suppose $p>2$ and $A_i$ is a sequence of connections converging to $A_\infty$ in $L^p$. Suppose furthermore that the connections in the sequence are gauge equivalent, i.e. there exist $k_i \in W^{1,p}(\Sig,K)$ such that $k_i(A_0)=A_i$.  Then, there exists $k_\infty \in W^{1,p}(\Sig,K)$ such that, after passing to a subsequence, $k_i \weakto k_\infty$ weakly in $W^{1,p}$. Further, $A_\infty=k_\infty(A_0)$ and so the limit is in the same gauge orbit as the sequence.
\end{lemma}
\begin{proof} This result is a slight variation of Lemma 4.3.5 in \cite{Venu}. The connections can be written as $A_i=\onD + a_i$, where $a_i \in \Om^1(\Sig,\k)$. Any compact Lie group injectively embeds into $U(N)$. After fixing such an embedding, we can use matrix representations for $a_i$, $k_i$ etc. The relation $k_iA_0=A_i$ can be re-written as
$$a_i=(dk_i)k_i^{-1}+k_ia_0k_i^{-1}.$$
Let $\Theta_i:=A_i-A_\infty$. Then,
\begin{align}\label{eq:gtb}
dk_i=-k_ia_0+a_\infty k_i +\Theta_ik_i.
\end{align}
The one-forms $A_0$, $A_\infty$ and $\Theta_i$ are bounded in $L^p$. We can assume that the norm on the space of $N \times N$ matrices is invariant under left and right multiplication by elements in $K$. Then, the right hand side of \eqref{eq:gtb} is bounded in $L^p$. Since $K$ is compact, $\Mod{k_i}_{W^{1,p}}<c$ for some constant $c$. After passing to a subsequence, $k_i \weakto k_\infty$ in $W^{1,p}$. So, the sequence $k_iA_0$, which is equal to the sequence $A_i$, converges to $k_\infty A_0$ weakly in $L^p$. The limit must be equal to the strong limit, which is $A_\infty$.
\end{proof}

%
\def\dbar{\leavevmode\hbox to 0pt{\hskip.2ex \accent"16\hss}d}
\providecommand{\bysame}{\leavevmode\hbox to3em{\hrulefill}\thinspace}
\providecommand{\MR}{\relax\ifhmode\unskip\space\fi MR }
\providecommand{\MRhref}[2]{%
  \href{http://www.ams.org/mathscinet-getitem?mr=#1}{#2}
}
\providecommand{\href}[2]{#2}

\end{document}